\documentclass[twoside,leqno,11pt]{amsart} 
\usepackage{amsfonts} 
\usepackage{amsmath} 
\usepackage{amscd} 
\usepackage{amssymb} 
\usepackage{amsthm} 
\usepackage{latexsym} 
\usepackage{color} 
\newtheorem{theorem}{Theorem} 
\newtheorem{lemma}{Lemma} 
\newtheorem{Proposition}{Proposition} 
 
\newtheorem{corollary}{Corollary} 
 
\numberwithin{equation}{section}

\title{Diophantine approximation with prime denominator in real quadratic function fields} 

\author{Stephan~Baier} 
\address{Stephan~Baier\\%
Ramakrishna Mission Vivekananda Educational and Research Institute\\%
Department of Mathematics\\%
G.\ T.\ Road, PO~Belur Math, Howrah, West Bengal~711202\\%
India} 
\email{stephanbaier2017@gmail.com} 
\urladdr{https://www.researchgate.net/profile/Stephan\_Baier2} 

\author{Esrafil Ali Molla} 
\address{Esrafil Ali Molla\\%
Ramakrishna Mission Vivekananda Educational and Research Institute\\%
Department of Mathematics\\%
G.\ T.\ Road, PO~Belur Math, Howrah, West Bengal~711202\\%
India} 
\email{esrafil.math@gmail.com} 

\date{\today}

\subjclass[2010]{Primary: 11J71, 11R44; Secondary: 11L20, 11N05, 11N13} 

\keywords{Distribution modulo one, function fields, distribution of prime ideals, Vaughan's identity, Diophantine approximation} 
\begin{document} 

\begin{abstract} In the thirties of the last century, I. M. Vinogradov proved that the inequality $||p\alpha||\le p^{-1/5+\varepsilon}$ has infinitely prime solutions $p$, where $||.||$ denotes the distance to a nearest integer. This result has subsequently been improved by many authors. In particular, Vaughan (1978) replaced the exponent $1/5$ by $1/4$ using his celebrated identity for the von Mangoldt function and a refinement of Fourier analytic arguments. The current record is due to Matom\"aki (2009) who showed the infinitude of prime solutions of the inequality $||p\alpha||\le p^{-1/3+\varepsilon}$. This exponent $1/3$ is considered the limit of the current technology. Recently, in \cite{BaMo}, the authors established an analogue of Matom\"aki's result for imaginary quadratic extensions of the function field $k=\mathbb{F}_q(T)$. In this paper, we consider the case of real quadratic extensions of $k$ of class number 1, for which we prove a function field analogue of Vaughan's above-mentioned result (exponent $\theta=1/4$). Our method uses versions of Vaughan's identity and the Dirichlet approximation theorem for function fields. The latter was established by Arijit Ganguly in the appendix to our previous paper \cite{BaMo} on the imaginary quadratic case. We also simplify arguments in the paper \cite{BM} on the same problem for real quadratic number fields by D. Mazumder and the first-named author. 
\end{abstract} 

\maketitle 

\tableofcontents 

\section{Introduction} 
Throughout this article, let $\varepsilon$ be an arbitrary but fixed positive real number. 

A fundamental result in Diophantine approximation is the following theorem due to Dirichlet. 

\begin{theorem} \label{Dirich} 
Given any real irrational $\alpha$, there are infinitely many pairs $(a,q)\in \mathbb{Z}\times \mathbb{N}$ of relatively prime integers such that 
$$ 
|\alpha - a/q | < q^{-2}. 
$$ 
\end{theorem} 

It is natural to study good rational approximations $a/q$ to irrational $\alpha$'s when $q$ is restricted to arithmetically interesting subsets of the positive integers, such as the set of primes. In this case, one may ask for which $\theta>0$ one can establish the infinitude of primes $p$ such that 
\begin{equation*} 
\left|\alpha-\frac{a}{p}\right| < p^{-(\theta+1)+\varepsilon} 
\end{equation*} 
for a suitable $a\in \mathbb{Z}$, or equivalently, 
\begin{equation}\label{eq:DirichletApprox:PrimeConstraint} 
\left|\left|p\alpha\right|\right| < p^{-\theta+\varepsilon}, 
\end{equation} 
where $||.||$ is the distance to a nearest integer. 
An overview of the interesting history of this problem was given in our recent paper \cite{BaMo}. Here we mention the following three milestones. In the thirties of the last century, I. M. Vinogradov established the infinitude of primes $p$ satisfying \eqref{eq:DirichletApprox:PrimeConstraint} for $\theta=1/5$. This result has subsequently been improved by many authors. In particular, Vaughan (1978) replaced the exponent $1/5$ by $1/4$ using his celebrated identity for the von Mangoldt function and a refinement of Fourier analytic arguments. The current record is due to Matom\"aki (2009) who was able to reach the exponent $\theta=1/3$ which is considered the limit of the current technology. 

Using Harman's sieve method, versions of these results for quadratic number fields have recently been established by Harman, the first-named author, Mazumder and Technau in several papers (see \cite{BM}, \cite{BMT}, \cite{BT}, \cite{Harman2019}), ultimately achieving an analogue of Harman's exponent $\theta=7/22$ for the original problem (see \cite{harman1996on-the-distribu}) in this setting. In \cite{BaMo}, the authors established an analogue of Matom\"aki's result (with exponent $\theta=1/3$) for the function field $k=\mathbb{F}_q(T)$ and its imaginary quadratic extensions. In this paper, we consider the case of real quadratic extensions $K$ of $k$. Our exponent is slightly weaker and corresponds to Vaughan's $\theta=1/4$. The reason for this weaker exponent is that we did not find a way to make our analytic arguments in \cite{BaMo} work in this setting. In \cite{BaMo}, we avoided sieve methods or Vaughan's identity and used the Riemann hypothesis for Hecke $L$-functions to approximate the number of prime elements in arithmetic progressions. This argument seems to fail here due to the infinitude of units in the integral closure of $\mathbb{F}_q[T]$ in real quadratic extensions $K$ of $k$. Therefore, we apply a version of Vaughan's identity for function fields in this article. Another essential tool is Dirichlet's approximation theorem for real quadratic function fields, which we derive from a general version of Dirichlet's approximation theorem by A. Ganguly. Our method is similar to that in the papers \cite{BM} and \cite{BMT} on real quadratic number fields, where we used Harman's sieve instead of Vaughan's identity. However, we manage to simplify some of the arguments in \cite{BM}, which allows us to sharpen the results therein. More precisely, we manage to remove a certain Diophantine condition which was termed ``good pairs" in \cite{BM}. In this paper, to avoid a number of technical subtleties, we consider only fields of class number 1. It is likely that with some efforts, our method can be extended to general number and function fields without class number restriction. Moreover, it may be possible to improve the exponent $1/4$ using a function field version of Harman's lower bound sieve. We leave these problems to future research.\\ \\ 
{\bf Acknowledgements.} The authors would like to thank the Ramakrishna Mission Vivekananda Educational and Research Insititute for an excellent work environment. The research of the second-named author was supported by a UGC NET grant under number NOV2017-424450. 

\section{Notations and preliminaries} 
The following notations and preliminaries are used throughout this article. 
\begin{itemize} 
\item Let $q=p^n$ be a prime power and $\mathbb{F}_q$ be the finite field with $q$ elements. Let $\mathbb{F}_q(T)_{\infty}$ be the completion of $\mathbb{F}_q(T)$ at $\infty$ 
(i.e. $\mathbb{F}_q(T)_{\infty}=\mathbb{F}_q((1/T))$). 
\item The absolute value $|. |=|. |_{\infty}$ at infinity on $\mathbb{F}_q(T)_{\infty}$ is defined as 
\begin{align*} 
\left| \sum_{i= -\infty}^n a_i T^i \right| = q^n \mbox{ if } a_n \not= 0. 
\end{align*} 
\item Consider the torus $\mathbb{T}= \mathbb{F}_q(T)_\infty/\mathbb{F}_q[T]$. A metric on $\mathbb{T}$ is given by 
\begin{align*} 
||x+\mathbb{F}_q[T]||:= \inf_{x'\in x+\mathbb{F}_q[T]}|x'| \quad (x\in \mathbb{F}_q(T)_{\infty}). 
\end{align*} 
Note that 
$\mathbb{T}$ is a compact Hausdorff space and for all $x+\mathbb{F}_q[T]\in \mathbb{T}$, we have $||x+\mathbb{F}_q[T]||\leq 1/q$. \item More in detail, if 
$$ 
x=\sum\limits_{i=-\infty}^{n} a_i T^i, 
$$ 
then 
$$ 
||x+\mathbb{F}_q[T]||=\left|\sum\limits_{i=-\infty}^{-1} a_iT^i \right|. 
$$ 
The sum on the right-hand side may be viewed as the fractional part of $x$. Clearly, 
$$ 
||x||=q^{k}, 
$$ 
where $k$ is the largest negative integer such that $a_k\not=0$. 
\item For $x\in \mathbb{F}_q(T)_{\infty}$, we also write 
$$ 
||x||:=||x+\mathbb{F}_q[T]||. 
$$ 
\item Let $\text{Tr} : \mathbb{F}_q \to \mathbb{F}_p$ be the trace map. A non-trivial additive character $E: \mathbb{F}_q \to \mathbb{C^\times}$ is defined by 
$$ E(x)= \exp{\left(\frac{2\pi i}{p} \text{Tr}(x)\right)}, $$ 
and an exponential map $e: \mathbb{F}_q(T)_\infty \to \mathbb{C^\times}$ is defined by 
$$ e\left( \sum_{i= -\infty}^n a_i T^i \right) = E(a_{-1}) .$$ 
This map $e$ is also a non-trivial additive character of $\mathbb{F}_q(T)_\infty$. 

\item Given $f=(f_1, f_2, \dots ,f_n ) \in \mathbb{F}_q(T)_\infty^n ,$ we define the additive character $ \Psi_f : \mathbb{F}_q(T)_\infty^n \to \mathbb{C}^\times$ as 
\begin{equation*} 
\Psi_f(g_1,g_2, \dots , g_n) = e(f_1g_1+ f_2g_2 + \dots + f_n g_n) 
= \prod_{i=1}^{n} e(f_i g_i) 
\end{equation*} 
for any $ (g_1,g_2, \dots , g_n) \in \mathbb{F}_q(T)_\infty^n .$ 

\item For $f=(f_1, f_2, \dots , f_n) \in \mathbb{F}_q(T)_\infty^n $, we define a metric $| . |$ on $ \mathbb{F}_q(T)_\infty^n $ as 
$$|f|= \sup \{ |f_1|, |f_2|, \dots , |f_n|\}. 
$$ 

\item Given $r>0$, the open ball $B_n(f,r)$ of radius $r$, centered at $f\in \mathbb{F}_q(T)_{\infty}^n$, is defined by 
$$B_n(f,r) := \{ g \in \mathbb{F}_q(T)_\infty^n : |g-f|< r \} .$$ 
Note that $B_n(0,r) = B(0,r)^n,$ where we set
$$
B(0,r):=B_1(0,r).
$$ 

\item On the locally compact topological space $\mathbb{F}_q(T)_\infty^n$, let $\mu$ be the unique Haar measure so that $$\mu(B_n(0,1))=1.$$  For any given locally constant function $ \mathfrak{F} : \mathbb{F}_q(T)_\infty^n \to \mathbb{C} $ with compact support, the Fourier transform $ \hat{\mathfrak{F}} $ is defined by 
\begin{align*} 
\hat{\mathfrak{F}}(f):= \int_{ \mathbb{F}_q(T)_\infty^n } \mathfrak{F}(g) \overline{\Psi_f (g) } \,d\mu(g) \mbox{ for any } f\in \mathbb{F}_q(T)_\infty^n. 
\end{align*}  

\item We define the M\"obius function on $\mathbb{F}_q[T]$ as 
$$ 
\mu(f):=\begin{cases} 0 & \mbox{ if } f \mbox{ is divisible by a square of a non-unit in } \mathbb{F}_q[T],\\ (-1)^{\omega(f)} & \mbox{ otherwise,} 
\end{cases} 
$$ 
where $\omega(f)$ is the number of non-associate irreducible factors of $f$. 
\item We define the von Mangoldt function on $\mathbb{F}_q[T]$ as 
$$ 
\Lambda(f):=\begin{cases} \log_q(|P|) & \mbox{ if } f=\epsilon P^r \mbox{ for some irreducible polynomial } P, \mbox{ a unit } \epsilon \mbox{ and } r\in \mathbb{N},\\ 0 & \mbox{ otherwise.} 
\end{cases} 
$$ 
\item If $f\not=0$, then we have the identities 
\begin{align*} 
\sum\limits_{\substack{g|f\\ g \mbox{\scriptsize \ monic }}} \Lambda(g)= \log_q(|f|) 
\end{align*} 
and 
\begin{align*} 
\sum\limits_{\substack{g|f\\ g \mbox{\scriptsize \ monic}}} \mu(g)=\begin{cases} 1 & \mbox{ if } f \in \mathbb{F}_q^{\times},\\ 0 & \mbox{ otherwise.} 
\end{cases} 
\end{align*} 
\end{itemize} 

\section{ Setup of the problem } 
In this section, we state our main result. To this end, we need to describe real quadratic extensions of $\mathbb{F}_q(T)$ and give a version of the Dirichlet approximation theorem for these fields. 

\subsection{Real quadratic function fields} 
Throughout the sequel, assume that $q$ is an odd prime power. Let $d \in \mathbb{F}_q[T]$ be a monic square-free polynomial of even degree. Then $d$ is the square of an element in $\mathbb{F}_q(T)_\infty$ which we denote by $\sqrt{d}$. Clearly, $\sqrt{d}$ is unique up to the sign. The field $K:= \mathbb{F}_q(T)(\sqrt{d})$ is called a real quadratic field extension of $ k:=\mathbb{F}_q(T)$. Below we introduce some more notations and recall well-known properties of these fields. 
\begin{itemize} 
\item Let $\mathbb{A}$ be the integral closure of $\mathbb{F}_q[T] $ in $K.$ Then we have 
$$ \mathbb{A} =\{ a + b \sqrt{d} : a,b \in \mathbb{F}_q[T]\}.$$ 
\item If $f = a + b\sqrt{d}\in K$, then we write 
\begin{equation} \label{RealImaginary} 
\Re(f)=a \quad \mbox{ and } \quad \Im(f)=b. 
\end{equation} 
We note that $\Re(f)$ and $\Im(f)$ are well-defined. 
\item If $f = a + b \sqrt{d}\in K$, then we define a {\bf Norm} on $K$ as 
$$ 
\mbox{\bf Norm}(f) = (a+b\sqrt{d})(a-b\sqrt{d})=a^2- b^2d \in k. 
$$ 
\item We denote the set of non-zero integral ideals in $\mathbb{A}$ by $\mathcal{I}$. 
\item The {\bf norm of an ideal} $\mathfrak{a}\in \mathcal{I}$ is defined as 
$$ 
\mathcal{N}(\mathfrak{a})= \sharp\left(\mathbb{A}/\mathfrak{a}\right). 
$$ 
If $\mathfrak{a}=(f)$ with $f\in \mathbb{A}$, then 
$$ 
\mathcal{N}(\mathfrak{a}) = |\mbox{\bf Norm}(f)|. 
$$ 
In this case, we also write 
$$ 
\mathcal{N}(f)=\mathcal{N}(\mathfrak{a}). 
$$ 
\item There are two automorphisms in $\mbox{Gal}(K:k)$, the identity and conjugation, given by 
$$ 
\sigma_1(a+b\sqrt{d}):=a+b\sqrt{d} 
$$ 
and 
$$ 
\sigma_2(a+b\sqrt{d}):=a-b\sqrt{d}. 
$$ 
We write 
\begin{align*} 
\sigma(K):=\{(\sigma_1(f),\sigma_2(f)) : f\in K\}. 
\end{align*} 
\item Correspondingly, the absolute value on $k$ extends to two absolute values on $K$, namely 
$$ 
|a+ b\sqrt{d}|_1 := |a+b\sqrt{d}| 
$$ 
and 
$$ 
|a+ b\sqrt{d}|_2 := |a-b\sqrt{d}| 
$$ 
where $|.|=|.|_{\infty}$ is the absolute value on $k_\infty$, defined earlier. 
\item The M\"obius function is defined on the non-zero integral ideals in $\mathbb{A}$ as the unique multiplicative function with the property that 
$$ 
\mu(\mathfrak{p}^r)=\begin{cases} -1 & \mbox{ if } r=1\\ 0 & \mbox{ if } r>1. \end{cases} 
$$ 
\item The Von Mangoldt function is defined on the non-zero integral ideals in $\mathbb{A}$ by 
$$ 
\Lambda(\mathfrak{f}):=\begin{cases} \log_q(\mathcal{N}(\mathfrak{p})) & \mbox{ if } \mathfrak{f}=\mathfrak{p}^r \mbox{ for some prime ideal } \mathfrak{p},\\ 0 & \mbox{ otherwise.} 
\end{cases} 
$$ 
\item The identities 
\begin{align*} 
\sum\limits_{\mathfrak{g}|\mathfrak{f} } \Lambda(\mathfrak{g})= \log_q( \mathcal{N}(\mathfrak{f})) 
\end{align*} 
and 
\begin{align*} 
\sum\limits_{\mathfrak{g}|\mathfrak{f}} \mu(\mathfrak{g})=\begin{cases} 1 & \mbox{ if } \mathfrak{f} = (1),\\ 0 & \mbox{ otherwise} 
\end{cases} 
\end{align*} 
hold for all $\mathfrak{f}\in \mathcal{I}$. 
\end{itemize} 

\subsection{Dirichlet approximation and main result} 
A general version of Dirichlet's approximation theorem for function fields was given in \cite[Theorem 5.3]{BaMo}. This implies the following theorem. 

\begin{theorem} \label{Dirichlet1} 
Assume that $K$ is a real quadratic extension of $k=\mathbb{F}_q(T)$ of class number 1. Then, if $(x_1, x_2) \in K_\infty^2 \setminus \sigma(K)$, there exists a constant $C >0$ depending only on $K$ such that 
\begin{equation} \label{Dirichl} 
\left|x_i-\frac{\sigma_i(p)}{\sigma_i(q)}\right| \le \frac{C}{\mathcal{N}(q)} \quad \mbox{ for } i=1,2 
\end{equation} 
for infinitely many elements $p/q \in K$ with $ (p,q) \in \mathbb{A} \times(\mathbb{A} \setminus \{0 \})$. 
\end{theorem} 

Now we are ready to state our main theorem below. 
\begin{theorem} \label{main theorem} 
Suppose that $q>2^{12}$ is an odd prime power. Let $K$ be a real quadratic field extension of $k$ of class number 1. Suppose that $(x_1, x_2) \in K_\infty^2 \setminus \sigma(K)$. Then, for every $\varepsilon>0$, there exist infinitely many elements $p/ \pi \in K $ with $p,\pi \in \mathbb{A}$ such that $\pi \in \mathbb{A}$ is a prime element and 
\begin{equation} \label{diorequired} 
\left|x_i-\frac{\sigma_i(p)}{\sigma_i(\pi)}\right| \le \mathcal{N}(\pi)^{-\frac{1}{2} - \frac{1}{8} + \log_q \sqrt{2}+\varepsilon} \quad \mbox{ for } i=1,2. 
\end{equation} 
\end{theorem} 

As demonstrated in \cite{BaMo}, the case when the class number is greater than 1 requires a good amount of extra work since we have no notion of a greatest common divisor of two elements of $\mathbb{A}$ in this case. In this paper, we avoid these subtleties, assuming throughout that the class number is 1 (i.e. $\mathbb{A}$ has unique factorization), which is conjecturally the case for ``many" real quadratic fields $K$ (for results in this direction, see \cite{Fri}, for example). 

\subsection{Basic approach} 
We conclude this section by describing our basic approach, which is inspired by the work in \cite{BM}. However, in the function field case, the situation is much simpler. In particular, in place of the smooth weight functions used in \cite{BM}, it here suffices to take characteristic functions of suitable sets. 

We define for $x\in K_\infty$, 
\begin{align} \label{Omegadef} 
\Omega_{\Delta}(x):= \begin{cases} 1 & \mbox{ if } |x| \le \Delta ,\\ 0 & \mbox{ otherwise.} \end{cases} 
\end{align} 
Further, we define a function $\omega : \mathcal{I} \to \mathbb{R}_{\ge 0} $ as 
\begin{equation} \label{Fdef} 
\omega(\mathfrak{a}):=\sum\limits_{p\in\mathbb{A}}\Omega_{\delta/q^{N/2}}\left(x_1-\frac{\sigma_1(p)}{\sigma_1(v)}\right)\Omega_{\delta/q^{N/2}}\left(x_2-\frac{\sigma_2(p)}{\sigma_2(v)}\right), 
\end{equation} 
where $v$ is any generator of $\mathfrak{a}$. We note that $\omega(\mathfrak{a})$ is well-defined, i.e. independent of the choice of the generator $v$. 

Our approach is to compare the two quantities 
\begin{align} \label{def of T(N)} 
{\mathcal{T}}(N):= \frac{\delta^2}{|\sqrt{d}|} \sum\limits_{\substack{ \mathfrak{a}\in \mathcal{I} \\ \mathcal{N}(\mathfrak{a})= q^N }}\Lambda(\mathfrak{a}) 
\end{align} 
and 
\begin{align} \label{Def of til(T(N))} 
\tilde{\mathcal{T}}(N):=\sum\limits_{\substack{ \mathfrak{a}\in \mathcal{I} \\ \mathcal{N}(\mathfrak{a})= q^N }}\Lambda(\mathfrak{a})\omega(\mathfrak{a}), 
\end{align} 
where $\Lambda$ is the von Mangoldt function for ideals. We aim to show that 
\begin{equation*} 
\tilde{\mathcal{T}}(N) \sim \mathcal{T}(N) 
\end{equation*} 
for a suitable sequence of $N$'s with $N \to \infty$. This, together with the prime number theorem for function fields, applied to the quantity $\mathcal{T}(N)$, will produce infinitely many elements $p/\pi$ of $K$ satisfying the desired Diophantine property in Theorem \ref{main theorem}. 

\section{Application of Poisson summation} 
In this section, we will use the Poisson summation formula for function fields to re-write the quantity $\tilde{\mathcal{T}}(N)$. The zero frequency will turn out precisely equal to the quantity $\mathcal{T}(N)$. Therefore, the difference of these two quantities consists of the non-zero frequency terms after our application of Poisson summation. We begin by providing the Poisson summation formula for function fields (see \cite[Theorem 4.2.1]{CaFr}). 

\begin{lemma}[Poisson Summation Formula] \label{Poissonformula} 
Let $\Lambda$ be a complete lattice in $k_\infty^n$ and let 
$$ \Lambda' = \{ g\in k_\infty^n : f\cdot g\in \mathbb{F}_q[T] \text{ for all } f \in \Lambda \}$$ 
be the lattice dual to $\Lambda$. Let $f:k_\infty^n \to \mathbb{C}$ be a function such that 
the function
$$ 
F(x)=\sum\limits_{a\in \Lambda}|f(x+a)| 
$$ 
is uniformly convergent on compact subsets of $k_{\infty}$ and 
$$ 
\sum\limits_{a'\in \Lambda'} |\hat{f}(a)| 
$$ 
is convergent. Then 
$$ 
\sum\limits_{a\in \Lambda}f(a) = \frac{1}{\mbox{\rm Cov}(\Lambda)} \sum\limits_{a' \in \Lambda'} \hat{f}(a'), 
$$ 
where $\mbox{\rm Cov} (\Lambda) $ is the covolume of $\Lambda$. 
\end{lemma} 

We recall that the covolume of $\Lambda$ equals $|\det M|$ if $M$ is a matrix of generators of $\Lambda$. 

For $\mathfrak{a}\in \mathcal{I}$ and a generator $v$ of $\mathfrak{a}$, let the matrix $A$ be defined as 
$$ 
A:= \begin{pmatrix} 
\frac{1}{\sigma_1(v)} & \frac{\sqrt{d}}{\sigma_1(v)} \\ 
\frac{1}{\sigma_2(v)}& -\frac{\sqrt{d}}{\sigma_2(v)} 
\end{pmatrix}. 
$$ 
Then $|\det(A)| = |\sqrt{d}|/ \mathcal{N}(\mathfrak{a})$ and 
$$ 
A^{-1}= \begin{pmatrix} 
\frac{\sigma_1(v)}{2} & \frac{\sigma_2(v)}{2} \\ 
\frac{\sigma_1(v)}{2 \sqrt{d}}& -\frac{\sigma_2(v)}{2 \sqrt{d}} 
\end{pmatrix}. 
$$ 
In our application, we shall apply the above Poisson summation formula to the lattice 
$$\Lambda =\left\{ A \begin{pmatrix} x \\ y \end{pmatrix} : (x, y ) \in \mathbb{F}_q[T]^2 \right\}$$ 
whose dual lattice is
$$\Lambda' =\left\{ (A^{-1})^T \begin{pmatrix} x \\ y \end{pmatrix} : (x, y ) \in \mathbb{F}_q[T]^2 \right\}.$$ 
To this end, we evaluate the Fourier transforms of characteristic functions of balls below. 

Let 
\begin{equation*} 
\chi_{B_2(0,\Delta)}\begin{pmatrix} x \\ y \end{pmatrix}:= \begin{cases} 1 & \mbox{ if } |x|, |y| \le \Delta ,\\ 0 & \mbox{ otherwise,} \end{cases} 
\end{equation*} 
be the characteristic function of the ball $B_2(0,\Delta)$. 
We recall that 
\begin{equation*} 
\chi_{B_2(0,\Delta)}\begin{pmatrix} x\\y \end{pmatrix} = \chi_{B(0,\Delta)}(x) \chi_{B(0,\Delta)}(y). 
\end{equation*} 
Clearly, the Fourier transform of $ \chi_{B_2(0,\Delta)} $ equals 
\begin{equation*} 
\hat{\chi}_{B_2(0,\Delta)} \begin{pmatrix} x\\y \end{pmatrix} = \hat{\chi}_{B(0,\Delta)}(x) \hat{\chi}_{B(0,\Delta)}(y),
\end{equation*} 
where $ \hat{\chi}_{B(0,\Delta)} $ is the Fourier transform of $ \chi_{B(0,\Delta)}$. We have the following result. 

\begin{lemma} \label{Fourierlemma} 
The Fourier transform of $\chi_{B(0,1)} $ equals 
$$ 
\hat{\chi}_{B(0,1)}(x) = \int\limits_{F_q(T)_\infty} \chi_{B(0,1)}(y) e(-xy) d\mu(y) = \chi_{B(0,1)}(x) . $$ 
\end{lemma} 

\begin{proof} 
This is \cite[Lemma 5.4]{BaierSingh}. 
\end{proof} 

As a corollary, we deduce the following generalization. 

\begin{corollary}\label{Chi1} 
The Fourier transform of $\chi_{B(0,q^N)} $ equals 
$$\hat{\chi}_{B(0,q^N)}(x) = q^N \chi_{B(0,q^{-N})}(x) .$$ 
\end{corollary} 
\begin{proof} 
We have 
\begin{equation*} 
\begin{split} 
\hat{\chi}_{B(0,q^N)}(x) &= \int\limits_{F_q(T)_\infty} \chi_{B(0,q^N)}(y) e(-xy) d\mu(y) \\ 
&= \int\limits_{F_q(T)_\infty} \chi_{B(0,1)}(T^{-N} y) e(-xy) d\mu(y) \\ 
&= q^N \int\limits_{F_q(T)_\infty} \chi_{B(0,1)}(z) e(-T^Nx z) d\mu(z) \\ 
&= q^N \hat{\chi}_{B(0,1)}(T^Nx). 
\end{split} 
\end{equation*} 
By Lemma \ref{Fourierlemma}, 
$$ \hat{\chi}_{B(0,1)}(T^Nx) = \chi_{B(0,1)}(T^Nx) .$$ 
Therefore, 
$$ \hat{\chi}_{B(0,q^N)}(x)= q^N \chi_{B(0,1)}(T^Nx) = q^N \chi_{B(0,q^{-N})}(x).$$ 
\end{proof} 

As another corollary of the above result, we establish the following. 

\begin{corollary}\label{G1} 
Let 
$$ 
g_{(x_1,x_2)} \begin{pmatrix} x\\y \end{pmatrix} := \chi_{B_2(0,\Delta)} \left( \begin{pmatrix} x_1\\y_2 \end{pmatrix} + \begin{pmatrix} x\\y \end{pmatrix} \right). 
$$ 
Then the Fourier transform of $g_{(x_1,x_2)} $ equals 
$$ 
\hat{g}_{(x_1,x_2)} \begin{pmatrix} \tilde{x}\\ \tilde{y} \end{pmatrix}=e(\tilde{x}x_1 + \tilde{y}x_2 ) \, \, \hat{\chi}_{B_2(0,\Delta)} \begin{pmatrix} \tilde{x}\\ \tilde{y} \end{pmatrix}. 
$$ 
\end{corollary} 
\begin{proof} We have 
\begin{equation*} 
\begin{split} 
\hat{g}_{(x_1,x_2)} \begin{pmatrix} \tilde{x}\\\tilde{y} \end{pmatrix} &= \int\limits_{F_q(T)_\infty ^2} g_{(x_1,x_2)} \begin{pmatrix} x\\y \end{pmatrix} e( - \tilde{x} x - \tilde{y} y) d\mu(x) d\mu(y) \\ 
&= \int\limits_{F_q(T)_\infty ^2} \chi_{B_2(0,\Delta)} \left( \begin{pmatrix} x_1\\y_2 \end{pmatrix} + \begin{pmatrix} x\\y \end{pmatrix} \right) e( - \tilde{x} x - \tilde{y} y) d\mu(x) d\mu(y) \\ 
&= \int\limits_{F_q(T)_\infty ^2} \chi_{B_2(0,\Delta)} \begin{pmatrix} x'\\y' \end{pmatrix} e( - \tilde{x} (x'- x_1) - \tilde{y} (y' -x_2)) d\mu(x') d\mu(y') \\ 
&= e(\tilde{x}x_1 + \tilde{y}x_2 )\int\limits_{F_q(T)_\infty ^2} \chi_{B_2(0,\Delta)} \begin{pmatrix} x'\\y' \end{pmatrix} e(-\tilde{x} x'- \tilde{y}y' ) d\mu(x') d\mu(y') \\ 
&=e(\tilde{x}x_1 + \tilde{y}x_2 ) \, \, \hat{\chi}_{B_2(0,\Delta)} \begin{pmatrix} \tilde{x}\\ \tilde{y} \end{pmatrix}. 
\end{split} 
\end{equation*} 
\end{proof} 

Now we are ready to prove the following result. 

\begin{Proposition} 
We have 
\begin{equation} \label{Facalc} 
\omega(\mathfrak{a})=\frac{\mathcal{N(\mathfrak{a})} \Delta^2}{|\sqrt{d}|}\cdot  \sum\limits_{p\in\mathbb{A}} e\left( \frac{x_1 \sigma_1(vp) - x_2\sigma_2(vp)}{2\sqrt{d}} \right) \chi_{B(0,\Delta^{-1})} \left( \frac{ \sigma_1(vp)}{2\sqrt{d}}\right) \chi_{B(0,\Delta^{-1} )} \left(\frac{\sigma_2(vp)}{2\sqrt{d}} \right), 
\end{equation} 
where 
\begin{equation} \label{deltadefi} 
\Delta:=\frac{\delta}{q^{N/2}}. 
\end{equation} 
\end{Proposition} 

\begin{proof} 
Using definition of $\omega(\mathfrak{a})$ in \eqref{Fdef}, we have 
\begin{equation*} 
\begin{split} 
\omega(\mathfrak{a})&=\sum\limits_{p\in\mathbb{A}}\Omega_{\Delta}\left(x_1-\frac{\sigma_1(p)}{\sigma_1(v)}\right)\Omega_{\Delta}\left(x_2-\frac{\sigma_2(p)}{\sigma_2(v)}\right)\\ 
&= \sum\limits_{(x,y)\in\mathbb{F}_q[T]^2} \chi_{B_2(0,\Delta)} \left( \begin{pmatrix} x_1\\ x_2 \end{pmatrix} - \begin{pmatrix} (x+y\sqrt{d})/\sigma_1(v) \\(x-y\sqrt{d})/\sigma_2(v) \end{pmatrix} \right)\\ 
&= \sum\limits_{(x,y)\in\mathbb{F}_q[T]^2} \chi_{B_2(0,\Delta)} \left( \begin{pmatrix} x_1\\ x_2 \end{pmatrix} - A \begin{pmatrix} x \\y \end{pmatrix} \right)\\ 
&= \sum\limits_{(x',y')\in \Lambda} \chi_{B_2(0,\Delta)} \left( \begin{pmatrix} x_1\\ x_2 \end{pmatrix} + \begin{pmatrix} x' \\y'\end{pmatrix} \right)\\ 
&= \sum\limits_{(x',y')\in \Lambda} g_{(x_1, x_2)} \begin{pmatrix} x' \\y'\end{pmatrix}. 
\end{split} 
\end{equation*} 
Applying Lemma \ref{Poissonformula}, the Poisson summation formula, to the last line above, we get 
\begin{equation*} 
\begin{split} 
\sum\limits_{(x',y')\in \Lambda} g_{(x_1, x_2)} \begin{pmatrix} x' \\y'\end{pmatrix}&= \frac{1}{\mbox{Cov}(\Lambda)} \sum\limits_{(x'',y'')\in \Lambda'} \hat{g}_{(x_1, x_2)} \begin{pmatrix} x'' \\y''\end{pmatrix} \\ 
&= \frac{1}{|\det(A)|} \sum\limits_{(x,y)\in\mathbb{F}_q[T]^2} \hat{g}_{(x_1, x_2)} \left( (A^{-1})^T \begin{pmatrix} x\\y \end{pmatrix} \right)\\ 
&= \frac{\mathcal{N(\mathfrak{a})}}{|\sqrt{d}|} \sum\limits_{(x,y)\in\mathbb{F}_q[T]^2} \hat{g}_{(x_1, x_2)}  \begin{pmatrix} \frac{x\sigma_1(v)}{2} + \frac{y\sigma_1(v)}{2\sqrt{d}}\\ \frac{x\sigma_2(v)}{2} - \frac{y\sigma_2(v)}{2\sqrt{d}} \end{pmatrix}.\\ 
\end{split} 
\end{equation*} 
Using Corollary \ref{G1}, the above becomes 
\begin{equation*} 
\begin{split} 
& \frac{\mathcal{N(\mathfrak{a})}}{|\sqrt{d}|} \sum\limits_{(x,y)\in\mathbb{F}_q[T]^2} e\left(x_1\left( \frac{x\sigma_1(v)}{2} + \frac{y\sigma_1(v)}{2\sqrt{d}} \right) +x_2\left(\frac{x\sigma_2(v)}{2} - \frac{y\sigma_2(v)}{2\sqrt{d}} \right) \right) \\ 
& \times \hat{\chi}_{B_2(0,\Delta)} \begin{pmatrix} \frac{x\sigma_1(v)}{2} + \frac{y\sigma_1(v)}{2\sqrt{d}}\\ \frac{x\sigma_2(v)}{2} - \frac{y\sigma_2(v)}{2\sqrt{d}} \end{pmatrix} \\ 
= & \frac{\mathcal{N(\mathfrak{a})}}{|\sqrt{d}|} \sum\limits_{(x,y)\in\mathbb{F}_q[T]^2} e\left( \frac{x_1 \sigma_1(v)}{2\sqrt{d}} (x\sqrt{d}+ y) + \frac{x_2\sigma_2(v)}{2\sqrt{d}} (x\sqrt{d} - y) \right) \\ 
& \times \hat{\chi}_{B(0,\Delta)} \left( \frac{ \sigma_1(v)}{2\sqrt{d}} (x\sqrt{d}+ y)\right) \hat{\chi}_{B(0,\Delta )} \left( \frac{\sigma_2(v)}{2\sqrt{d}} (x\sqrt{d} - y) \right) .\\ 
\end{split} 
\end{equation*} 
Using Corollary \ref{Chi1}, the above takes the form
\begin{equation*} 
\begin{split} 
& \frac{\mathcal{N(\mathfrak{a})} \Delta^2}{|\sqrt{d}|} \sum\limits_{(x,y)\in\mathbb{F}_q[T]^2} e\left( \frac{x_1 \sigma_1(v)}{2\sqrt{d}} (x\sqrt{d}+ y) + \frac{x_2\sigma_2(v)}{2\sqrt{d}} (x\sqrt{d} - y) \right) \\ 
& \quad \times \chi_{B(0,\Delta^{-1})} \left( \frac{ \sigma_1(v)}{2\sqrt{d}} (x\sqrt{d}+ y)\right) \chi_{B(0,\Delta^{-1} )} \left( \frac{\sigma_2(v)}{2\sqrt{d}} (x\sqrt{d} - y) \right)\\ 
&= \frac{\mathcal{N(\mathfrak{a})} \Delta^2}{|\sqrt{d}|} \sum\limits_{p\in\mathbb{A}} e\left( \frac{x_1 \sigma_1(vp) - x_2\sigma_2(vp)}{2\sqrt{d}} \right) \cdot
\chi_{B(0,\Delta^{-1})} \left( \frac{ \sigma_1(vp)}{2\sqrt{d}}\right) \chi_{B(0,\Delta^{-1} )} \left(\frac{\sigma_2(vp)}{2\sqrt{d}} \right). 
\end{split} 
\end{equation*} 
Combining everything, we obtain the desired identity. 
\end{proof} 

We note that the contribution of $p=0$ to the right-hand side of \eqref{Facalc} equals 
$$ 
\frac{\mathcal{N}(\mathfrak{a})\Delta^2}{|\sqrt{d}|}= \frac{\delta^2}{|\sqrt{d}|} 
$$ 
if $\mathcal{N}(\mathfrak{a})=q^N$. 
Therefore, 
\begin{equation} \label{T(N)-til(T(N))} 
\begin{split} 
\tilde{\mathcal{T}}(N)-\mathcal{T}(N)= & \sum\limits_{\substack{ \mathfrak{a}\in \mathcal{I} \\ \mathcal{N}(\mathfrak{a})= q^N }} \Lambda(\mathfrak{a})\left(\omega(\mathfrak{a})-\frac{\delta^2}{|\sqrt{d}|}\right)\\ 
= & \frac{\delta^2}{|\sqrt{d}|} \sum\limits_{\substack{ \mathfrak{a}\in \mathcal{I} \\ \mathcal{N}(\mathfrak{a})= q^N }} \Lambda(\mathfrak{a}) 
\sum\limits_{p\in\mathbb{A}\setminus \{0\}} e\left( \frac{x_1 \sigma_1(vp) - x_2\sigma_2(vp)}{2\sqrt{d}} \right) 
\times\\ & \chi_{B(0,\Delta^{-1})} \left( \frac{ \sigma_1(vp)}{2\sqrt{d}}\right)\cdot \chi_{B(0,\Delta^{-1} )} \left(\frac{\sigma_2(vp)}{2\sqrt{d}} \right). 
\end{split} 
\end{equation} 
In the next sections, our goal is to prove the following theorem, from which we will deduce our main result, Theorem \ref{main theorem}. 

\begin{theorem} \label{Main Theorem 2} 
Let $q>2^{12}$ be an odd prime power. Suppose that $(x_1, x_2) \in K_\infty^2 \setminus \sigma(K)$ and $\varepsilon >0 $ is sufficiently small. Then there exists an infinite sequence of positive integers $N$ such that 
$$ 
\tilde{\mathcal{T}}(N)-\mathcal{T}(N) \ll_{\varepsilon} \delta^2 q^{(1-\varepsilon)N} 
$$ 
provided that 
\begin{equation} \label{deltaassump} 
q^{-(1/8 -\log_q \sqrt{2}-2\varepsilon)N}\le \delta\le q^{-(\log_q 2 +2\varepsilon)N}. 
\end{equation} 
\end{theorem} 

\section{Vaughan's identity} 
In this section, we provide versions of Vaughan's identity for function fields which will then be used to reduce the sum in the last line of \eqref{T(N)-til(T(N))} involving $\Lambda(\mathfrak{a})$ to so-called type-I and type-II bilinear sums. In the following, we formulate a basic version of this identity for $\mathbb{F}_q[T]$, which may be useful for other applications too. 

\begin{lemma} \label{Vaughanfirst} 
Let $f \in \mathbb{F}_q[T]\setminus\{0\}$ be a monic polynomial and $\alpha, \beta$ be any non- negative integers such that $\beta < \deg(f)$. Then 
\begin{equation} \label{Lamda1} 
\Lambda (f)= \sideset{}{^\flat} \sum\limits_{\substack{ l|f \\ \deg(l) \le \alpha}}\mu(l)\log_q(|f/l|) - \sideset{}{^\flat}\sum\limits_{\substack{ ml|f \\ \deg(m) \le \alpha \\ \deg(l) \le \beta}} \sum\mu(m) \Lambda(l) + \sideset{}{^\flat}\sum\limits_{\substack{ dm|f \\ \deg(d) > \alpha\\ \deg(m) > \beta}} \sum\mu(d) \Lambda(m), 
\end{equation} 
where the superscript ``$\flat$" indicates that the summations are restriced to monic polynomials only. 
\end{lemma} 
\begin{proof} 
This is a function field version of \cite[Proposition 13.5]{IwKo} and can be proved in an analogous way. 
\end{proof} 

We re-write this identity as follows. 

\begin{corollary} 
Under the conditions of Lemma \ref{Vaughanfirst}, we have 
$$ 
\Lambda (f)= a_1(f) + a_2(f)+ a_3(f), 
$$ 
where 
$$ 
a_1(f):=- \sideset{}{^\flat}\sum\limits_{\substack{ nml=f \\ \deg(m) \le \alpha\\ \deg(l) \le \beta }} \mu(m) \Lambda(l), 
$$ 
$$ 
a_2(f):=\sideset{}{^\flat}\sum\limits_{\substack{ ln=f \\ \deg(l) \le \alpha}}\mu(l)\log_q(|n|) 
$$ 
and 
$$ a_3(f):= -\sideset{}{^\flat}\sum\limits_{\substack{ mn=f \\ \deg(m) > \beta\\ \deg(n) >0}} \Lambda(m) \sideset{}{^\flat}\sum\limits_{\substack{d|n \\ \deg(d) \le \alpha}} \mu(d). 
$$ 
\end{corollary} 
\begin{proof} 
Obviously, $a_1(f)$ and $a_2(f)$ equal the second and first sums on the right side of \eqref{Lamda1}, respectively. The third sum can be expressed in the form 
\begin{equation*} 
a_3(f) = \sideset{}{^\flat}\sum\limits_{\substack{ nm=f \\ \deg(m) > \beta}} \Lambda(m) \sideset{}{^\flat}\sum\limits_{\substack{\deg(d) > \alpha \\ d|n}}\mu(d) 
= -\sideset{}{^\flat}\sum\limits_{\substack{ nm=f \\ \deg(m) > \beta\\ \deg(n) >0}} \Lambda(m) \sideset{}{^\flat}\sum\limits_{\substack{d|n \\ \deg(d) \le \alpha}} \mu(d). 
\end{equation*} 
This completes the proof. 
\end{proof} 

For ideals in $\mathbb{A}$, the following identity can be established in essentially the same way. 

\begin{lemma} \label{Vaughanideals} 
Let $\mathfrak{f}$ be a non-zero integral ideal in $\mathbb{A}$ and $\alpha, \beta$ be positive integers such that $\beta < \mathcal{N}(f)$. Then 
$$ 
\Lambda(\mathfrak{f})= a_1(\mathfrak{f}) + a_2(\mathfrak{f})+ a_3(\mathfrak{f}), 
$$ 
where 
$$ 
a_1(\mathfrak{f}):=- \sum\limits_{\substack{ \mathfrak{n}\mathfrak{m}\mathfrak{l}=\mathfrak{f}\\ \mathcal{N}(\mathfrak{m}) \le \alpha\\ \mathcal{N}(\mathfrak{l}) \le \beta}} \mu(\mathfrak{m}) \Lambda(\mathfrak{l}), 
$$ 
$$ 
a_2(\mathfrak{f}):=\sum\limits_{\substack{\mathfrak{l} \mathfrak{n}=\mathfrak{f} \\ \mathcal{N}(\mathfrak{l}) \le \alpha}}\mu(\mathfrak{l})\log_q( \mathcal{N}(\mathfrak{n})) 
$$ 
and 
$$ a_3(\mathfrak{f}):= -\sum\limits_{\substack{ \mathfrak{m}\mathfrak{n}=\mathfrak{f} \\ \mathcal{N}(\mathfrak{m}) > \beta \\ \mathcal{N}(\mathfrak{n}) >1}} \Lambda(\mathfrak{m}) \sum\limits_{\substack{\mathfrak{d} |\mathfrak{n} \\ \mathcal{N}(\mathfrak{d}) \le \alpha}} \mu(\mathfrak{d}). 
$$ 
\end{lemma} 

Let $\mathcal{F} :\mathcal{I} \to \mathbb{C}$ be a function such that $\mathcal{F}(f)=0$ if $\mathcal{N}(f)\le \beta$. Then it follows from Lemma \ref{Vaughanideals} that 
\begin{equation*} 
\sum\limits_{\mathcal{N}(\mathfrak{f})\le X} \Lambda(\mathfrak{f}) \mathcal{F}(\mathfrak{f}) =S_1 + S_2 + S_3, 
\end{equation*} 
where $$ S_i= \sum\limits_{\mathcal{N}(\mathfrak{f})\le X} \mathcal{F}(\mathfrak{f}) a_i(\mathfrak{f}).$$ 
In the following, we transform $S_1$, $S_2$ and $S_3$ into bilinear sums and bound them suitably. We begin by writing 
\begin{equation*} 
\begin{split} 
S_1 &= - \sum\limits_{\mathcal{N}(\mathfrak{f}) \le X} \sum\limits_{\substack{\mathfrak{n}\mathfrak{l}\mathfrak{m} =\mathfrak{f} \\ \mathcal{N}(\mathfrak{l}) \le \beta \\ \mathcal{N}(\mathfrak{m}) \le \alpha }} \Lambda(\mathfrak{l}) \mu(\mathfrak{m}) \mathcal{F}(\mathfrak{f}) \\ 
&= - \sum\limits_{ \mathcal{N}(\mathfrak{t})\le \alpha\beta} \Bigg( \sum\limits_{\substack{ \mathfrak{l}\mathfrak{m} =\mathfrak{t} \\ \mathcal{N}(\mathfrak{l})\le \beta \\ \mathcal{N}(\mathfrak{m})\le \alpha}} \Lambda (\mathfrak{l}) \mu(\mathfrak{m})\Bigg) \sum\limits_{\substack{ \mathcal{N}(\mathfrak{n}) \le X/\mathcal{N}(\mathfrak{t})}} \mathcal{F}(\mathfrak{t}\mathfrak{n}) \\ 
&\ll \log_q(\alpha\beta) \sum\limits_{ \mathcal{N}(\mathfrak{t})\le \alpha\beta} \Bigg| \sum\limits_{\substack{ \mathcal{N}(\mathfrak{n}) \le X/\mathcal{N}(\mathfrak{t})}} \mathcal{F}(\mathfrak{t}\mathfrak{n}) \Bigg|. 
\end{split} 
\end{equation*} 
The sum $S_2$ satisfies 
\begin{equation*} 
\begin{split} 
S_2 &= \sum\limits_{\mathcal{N}(\mathfrak{f}) \le X} \Bigg( \sum\limits_{\substack{\mathfrak{l}\mathfrak{n}= \mathfrak{f} \\ \mathcal{N}(\mathfrak{l})\le \alpha }}\mu(\mathfrak{l}) \log_q \mathcal{N}(\mathfrak{n}) \Bigg) \mathcal{F}(\mathfrak{f}) \\ 
&= \sum\limits_{\mathcal{N}(\mathfrak{l}) \le \alpha} \mu(\mathfrak{l}) \sum\limits_{\substack{ \mathcal{N}(\mathfrak{n})\le X/ \mathcal{N}(\mathfrak{l})}} \mathcal{F}(\mathfrak{l}\mathfrak{n}) \cdot \log_q \mathcal{N}(\mathfrak{n}) \\ 
&= \frac{1}{\log q} \sum\limits_{\mathcal{N}(\mathfrak{l}) \le \alpha} \mu(\mathfrak{l}) \sum\limits_{\substack{ \mathcal{N}(\mathfrak{n})\le X/ \mathcal{N}(\mathfrak{l})}} \mathcal{F}(\mathfrak{l}\mathfrak{n}) \cdot \int\limits_{1}^{\mathcal{N}(\mathfrak{n})}\frac{d\mu(w)}{w}\\ 
&= \frac{1}{\log q} \int\limits_{1}^{X} \Bigg( \sum\limits_{\mathcal{N}(\mathfrak{l}) \le \alpha} \mu(\mathfrak{l}) \sum\limits_{\substack{ w\le \mathcal{N}(\mathfrak{n})\le X/ \mathcal{N}(\mathfrak{l})}} \mathcal{F}(\mathfrak{l}\mathfrak{n}) \Bigg) \frac{d\mu(w)}{w} \\ 
&\ll (\log_q X) \sum\limits_{\mathcal{N}(\mathfrak{l}) \le \alpha} \sup\limits_{1\le w\le X/\mathcal{N}(l)} \Bigg| \sum\limits_{\substack{ \mathcal{N}(\mathfrak{n})\le X/\mathcal{N}(l)}} \mathcal{F}(\mathfrak{l}\mathfrak{n}) - \sum\limits_{\substack{ \mathcal{N}(\mathfrak{n})\le w}} \mathcal{F}(\mathfrak{l}\mathfrak{n})\Bigg|\\ 
&\ll (\log_q X) \sum\limits_{\mathcal{N}(\mathfrak{l}) \le \alpha} \sup\limits_{1\le w\le X/\mathcal{N}(l)} \Bigg| \sum\limits_{\substack{ \mathcal{N}(\mathfrak{n})\le w}} \mathcal{F}(\mathfrak{l}\mathfrak{n})\Bigg|. 
\end{split} 
\end{equation*} 
Next, we set 
\begin{equation} \label{Gdef} 
H(\mathfrak{n}):=\sum\limits_{\substack{\mathfrak{d}|\mathfrak{n}\\\mathcal{N}(\mathfrak{d})\le \alpha }}\mu(\mathfrak{d}) 
\end{equation} 
and note that $H(\mathfrak{n}) = 0 \text{ for } 1<\mathcal{N}(\mathfrak{n})\le \alpha. $ 
Therefore, the sum $S_3$ takes the form 
\begin{equation*} 
\begin{split} 
S_3 &= - \sum\limits_{\mathcal{N}(\mathfrak{f}) \le X } \quad \Bigg( \sum\limits_{\substack{ \mathfrak{m}\mathfrak{n}=\mathfrak{f} \\ \mathcal{N}(\mathfrak{n}) > \alpha \\ \mathcal{N}(\mathfrak{m}) > \beta}} \Lambda(\mathfrak{m}) H(\mathfrak{n}) \Bigg) \mathcal{F}(\mathfrak{f}) \\ 
& = - \sum\limits_{\beta < \mathcal{N}(\mathfrak{m}) \le X/ \alpha} \quad \sum\limits_{\alpha < \mathcal{N}(\mathfrak{n}) \le X/ \mathcal{N}(\mathfrak{m})} \Lambda(\mathfrak{m})H(\mathfrak{n}) \mathcal{F}(\mathfrak{m}\mathfrak{n}). 
\end{split} 
\end{equation*} 
Combining the above estimates, we arrive at the following proposition. 

\begin{Proposition} \label{Vaughanbreaking} 
Let $X\ge q$ be an integer and $\alpha, \beta$ be positive integers such that $\alpha\beta < X$. Let $\mathcal{F}$ be any complex-valued function on the integral ideals in 
$\mathbb{A}$ such that $\mathcal{F}(f)=0$ if $\mathcal{N}(f)\le \beta$. Then we have 
\begin{equation*} 
\sum\limits_{\mathcal{N}(\mathfrak{f})\le X} \Lambda(\mathfrak{f}) \mathcal{F}(\mathfrak{f}) \ll T_{I} + T_{II} 
\end{equation*} 
with 
\begin{equation} \label{typeIsum} 
T_I:=\log_q(X) \sum\limits_{\mathcal{N}(\mathfrak{m}) \le \alpha\beta} \sup\limits_{1\le w\le X/\mathcal{N}(m)} \Bigg| \sum\limits_{\substack{\mathcal{N}(\mathfrak{n})\le w}} \mathcal{F}(\mathfrak{m}\mathfrak{n}) \Bigg| 
\end{equation} 
and 
\begin{equation} \label{typeIIsum} 
T_{II} := \Bigg| \sum\limits_{\beta < \mathcal{N}(\mathfrak{m}) \le X/ \alpha} \quad \sum\limits_{\alpha < \mathcal{N}(\mathfrak{n}) \le X/ \mathcal{N}(\mathfrak{m})} \Lambda(\mathfrak{m})H(\mathfrak{n}) \mathcal{F}(\mathfrak{m}\mathfrak{n})\Bigg|, 
\end{equation} 
where $H(\mathfrak{n})$ is defined as in \eqref{Gdef}. 
\end{Proposition} 

We note that 
\begin{equation} \label{LambdaGestis} 
\Lambda(\mathfrak{m})\le \log_q \mathcal{N}(\mathfrak{m}) \le \log_q X \quad \mbox{and} \quad |H(\mathfrak{n})|\le 2^{\log_q \mathcal{N}(\mathfrak{m})}\le X^{\log_q 2}. 
\end{equation} 
According to usual terminology, the sum $T_{I}$ above is called type-I sum, and the sum $T_{II}$ is called type-II sum. 

In the following, we apply Proposition \ref{Vaughanbreaking} to the situation when $X:=q^N$ and the function $\mathcal{F}$ satisfies $\mathcal{F}(\mathfrak{a})=0$ if $\mathcal{N}(\mathfrak{a})\not=q^N$ and 
\begin{equation} \label{Fdefi} 
\mathcal{F}(\mathfrak{a})=\frac{\delta^2}{|\sqrt{d}|} \sum\limits_{p\in\mathbb{A}\setminus \{0\}} e\left( \frac{x_1 \sigma_1(vp) - x_2\sigma_2(vp)}{2\sqrt{d}} \right) 
\chi_{B(0,\Delta^{-1})} \left( \frac{ \sigma_1(vp)}{2\sqrt{d}}\right) \chi_{B(0,\Delta^{-1} )} \left(\frac{\sigma_2(vp)}{2\sqrt{d}} \right) 
\end{equation} 
if $\mathcal{N}(\mathfrak{a})=q^N$. Here we recall that $v$ is any generator of $\mathfrak{a}$ and note that 
\begin{equation} \label{anote} 
\frac{\delta^2}{|\sqrt{d}|}= \frac{q^N\Delta^2}{|\sqrt{d}|} 
\end{equation} 
by \eqref{deltadefi}. In the above situation, recalling \eqref{T(N)-til(T(N))}, the said Proposition \ref{Vaughanbreaking} gives 
\begin{equation} \label{diffi} 
\Tilde{\mathcal{T}}(N)-\mathcal{T}(N)\ll T_I+T_{II}, 
\end{equation} 
where the type-I sum takes the form 
\begin{equation}\label{typeIform} 
T_I=\log_q(\alpha\beta) \sum\limits_{\substack{ \mathfrak{t}\in \mathcal{I} \\ \mathcal{N}(\mathfrak{t})\le \alpha\beta }} \Bigg| \sum\limits_{\substack{ \mathfrak{n}\in \mathcal{I} \\ \mathcal{N}(\mathfrak{n}) = q^N/\mathcal{N}(\mathfrak{t})}} \mathcal{F}(\mathfrak{t}\mathfrak{n}) \Bigg|
\end{equation} 
(note that $\mathcal{F}(\mathfrak{tn})=0$ if $\mathcal{N}(\mathfrak{n})<q^N/\mathcal{N}(\mathfrak{t})$), and the type-II sum takes the form 
\begin{equation} \label{typeIIform} 
T_{II}= \Big| \sum\limits_{\substack{\mathfrak{m} \in \mathcal{I} \\ \beta < \mathcal{N}(\mathfrak{m}) \le q^N/ \alpha }} \Lambda(\mathfrak{m}) \quad \sum\limits_{\substack{\mathfrak{n} \in \mathcal{I} \\\alpha < \mathcal{N}(\mathfrak{n}) = q^N/ \mathcal{N}(\mathfrak{m})} } H(\mathfrak{n}) \mathcal{F}(\mathfrak{n}\mathfrak{m})\Big|. 
\end{equation} 

\section{Transformation of the type-I sum} \label{transformationTI} 
The above equation \eqref{typeIform} implies 
\begin{equation*} 
T_I=\log_q(\alpha\beta) \sum\limits_{\substack{i\in \mathbb{N}_0 \\ q^i \le \alpha \beta }} \quad \sum\limits_{\substack{ \mathfrak{t}\in \mathcal{I} \\ \mathcal{N}(\mathfrak{t}) =q^i}} \Bigg| \sum\limits_{\substack{ \mathfrak{n}\in \mathcal{I} \\ \mathcal{N}(\mathfrak{n}) = q^{N-i}}} \mathcal{F}(\mathfrak{t}\mathfrak{n}) \Bigg|, 
\end{equation*} 
where $\mathbb{N}_0 = \mathbb{N} \cup \{0\}$. To simplify matters, we introduce a set $S(\mathfrak{a})$ associated to $\mathfrak{a}\in \mathcal{I}$ as follows. By the Dirichlet unit theorem for real quadratic extensions of $\mathbb{F}_q(T)$ (see \cite[Proposition 14.2]{Ros}), the group $\mathcal{U}$ of units of $\mathbb{A}$ equals 
\begin{equation} \label{units} 
\mathcal{U}=\{ \epsilon u^r : \epsilon \in \mathbb{F}_q^\times, r\in \mathbb{Z} \} 
\end{equation} 
for some unit $u$ with $|u|>1$, called fundamental unit. For $\mathfrak{a}$ an integral ideal in $\mathbb{A}$, we define 
\begin{equation}\label{cardS} 
S(\mathfrak{a}):= \{ v \in \mathbb{A} : (v)=\mathfrak{a} , \,\, \mathcal{N}(\mathfrak{a})^{1/2}|u|^{-1/2} < |v| \le \mathcal{N}(\mathfrak{a})^{1/2}|u|^{1/2} \}. 
\end{equation} 
By \eqref{units}, 
all generators of $\mathfrak{a}$ are of the form $\epsilon u^r v_0$ for a fixed $v_0 \in \mathbb{A}$. 
Thus, the cardinality of $\sharp S(\mathfrak{a})$ is exactly 
\begin{equation} \label{Sidentity} 
\sharp S(\mathfrak{a}) = \sharp\mathbb{F}_q^\times = q-1. 
\end{equation} 

Using the definitions of $\mathcal{F}(\mathfrak{a})$ and $\mathcal{S}(\mathfrak{a})$ above together with the identities \eqref{anote} and \eqref{Sidentity}, it follows that 
\begin{equation*} 
T_I=\frac{q^N \Delta^2\log_q(\alpha\beta)}{|\sqrt{d}|(q-1)^2} \sum\limits_{\substack{i\in \mathbb{N}_0 \\ q^i \le \alpha \beta }} \quad \sum\limits_{\substack{ t \in \mathbb{A} \\ \mathcal{N}(\mathfrak{t}) =q^i \\ t \in S(\mathfrak{t})}} \Bigg| \sum\limits_{\substack{ n\in \mathbb{A} \\ \mathcal{N}(\mathfrak{n}) = q^{N-i} \\ tn \in S(\mathfrak{t}\mathfrak{n}) }} \quad \sum\limits_{p \in \mathbb{A}\setminus \{0\}} \mathcal{E} (tnp) E(tnp)\Bigg|\\ 
\end{equation*} 
with the conventions that $(t)=\mathfrak{t}$ and $(n)=\mathfrak{n}$ and 
\begin{equation} \label{EEdefi} 
\mathcal{E}(l):= e\left( \frac{x_1 \sigma_1(l) - x_2\sigma_2(l)}{2\sqrt{d}} \right), \quad 
E(l):= \chi_{B(0,\Delta^{-1})} \left( \frac{ \sigma_1(l)}{2\sqrt{d}}\right) \chi_{B(0,\Delta^{-1} )} \left(\frac{\sigma_2(l)}{2\sqrt{d}} \right). 
\end{equation} 
Since 
\begin{equation} \label{Eident} 
E(tnp)=\begin{cases} 1 & \mbox{ if } \max\left\{|\frac{ \sigma_1(p) \sigma_1(tn)}{ 2\sqrt{d}}|, |\frac{ \sigma_2(p) \sigma_2(tn)}{ 2\sqrt{d}}|\right\} \le \Delta^{-1} ,\\ 0 & \mbox{ otherwise,} \end{cases} 
\end{equation} 
we deduce that 
\begin{equation} \label{TIequation} 
\begin{split} 
T_I &=\frac{ q^N \Delta^2 \log_q(\alpha\beta)}{|\sqrt{d}|(q-1)^2} \sum\limits_{\substack{i\in \mathbb{N} \\ q^i \le \alpha \beta }} \quad \sum\limits_{\substack{ t \in \mathbb{A} \\ \mathcal{N}(\mathfrak{t}) =q^i \\ t \in S(\mathfrak{t})}} \Bigg| \sum\limits_{\substack{ n\in \mathbb{A} \\ \mathcal{N}(\mathfrak{n}) = q^{N-i} \\ tn \in S(\mathfrak{t}\mathfrak{n}) \\ |\sigma_{1,2}(n)|\le |\sqrt{d}|/(\Delta|\sigma_{1,2}(tp)|)} } \quad \sum\limits_{\substack{p \in \mathbb{A}\setminus \{0\}}} \mathcal{E} (tnp)\Bigg| \\ 
&= \frac{ q^N \Delta^2 \log_q(\alpha\beta)}{|\sqrt{d}|(q-1)^2} \sum\limits_{\substack{i\in \mathbb{N} \\ q^i \le \alpha \beta }} T_I(i), 
\end{split} 
\end{equation} 
where 
\begin{equation} \label{TI(i)defi} 
T_I(i):= 
\sum\limits_{\substack{ t \in \mathbb{A} \\ \mathcal{N}(\mathfrak{t}) =q^i \\ t \in S(\mathfrak{t})}} \Bigg| \sum\limits_{\substack{p \in \mathbb{A}\setminus \{0\}}} \quad \sum\limits_{\substack{ n\in \mathbb{A} \\ \mathcal{N}(\mathfrak{n}) = q^{N-i} \\ tn \in S(\mathfrak{t}\mathfrak{n}) \\ |\sigma_{1,2}(n)|\le |\sqrt{d}|/(\Delta|\sigma_{1,2}(tp)|) }} \mathcal{E} (tnp)\Bigg|. 
\end{equation} 

Using the definition of $\mathcal{S}(\mathfrak{a})$ in \eqref{cardS}, if $tn \in S(\mathfrak{t}\mathfrak{n})$, then 
\begin{equation} \label{presigma} 
\mathcal{N}(\mathfrak{tn})^{1/2} |u|^{-1/2} < |tn| \le \mathcal{N}(\mathfrak{tn})^{1/2} |u|^{1/2}. 
\end{equation} 
Since $\sigma_1(tn)= tn$ and $ | \sigma_1(tn ) \sigma_2(tn)| = \mathcal{N}(\mathfrak{tn}) =q^N$, \eqref{presigma} implies 
\begin{equation}\label{sigma1} 
q^{N/2} |u|^{-1/2} < |\sigma_1(tn)| \le q^{N/2} |u|^{1/2} 
\end{equation} 
and 
\begin{equation}\label{sigma2} 
q^{N/2}|u|^{-1/2} \le |\sigma_2(tn)| < q^{N/2} |u|^{1/2}. 
\end{equation} 
Similarly, if $t\in \mathcal{S}(\mathfrak{t})$ in \eqref{cardS} and $\mathcal{N}(\mathfrak{t})=q^i$, then we have 
\begin{equation} \label{sigma1t} 
q^{i/2} |u|^{-1/2} < |\sigma_{1}(t)| \le q^{i/2} |u|^{1/2} 
\end{equation} 
and 
\begin{equation} \label{sigma2t} 
q^{i/2} |u|^{-1/2} \le |\sigma_{2}(t)| < q^{i/2} |u|^{1/2}. 
\end{equation} 
It follows that 
\begin{equation} \label{sigma12n} 
q^{(N-i)/2} |u|^{-1} < |\sigma_{1,2}(n)| <q^{(N-i)/2} |u|. 
\end{equation} 
Hence, the summation condition $|\sigma_{1,2}(n)|\le |\sqrt{d}|/(\Delta|\sigma_{1,2}(tp)|)$ 
implies the inequality
\begin{equation} \label{sigma12tp} 
|\sigma_{1,2}(tp)|\le |\sqrt{d}u|\Delta^{-1} q^{(i-N)/2}, 
\end{equation} 
which we shall use later on. 

Next, we transform the inner-most sum 
\begin{equation} \label{Vdef} 
V:=\sum\limits_{\substack{ n\in \mathbb{A} \\ \mathcal{N}(\mathfrak{n}) = q^{N-i} \\ tn \in S(\mathfrak{t}\mathfrak{n}) \\ |\sigma_{1,2}(n)|\le |\sqrt{d}|/(\Delta |\sigma_{1,2}(tp)|) }} \mathcal{E} (tnp) 
\end{equation} 
on the right-hand side of \eqref{TI(i)defi}. From \eqref{sigma1} and \eqref{sigma2}, it follows that 
\begin{equation} \label{Veq} 
\begin{split} 
V &= \sum\limits_{\substack{ n\in \mathbb{A} \\ q^{N/2} |u|^{-1/2}/|\sigma_1(t)| < |\sigma_1(n)| \le q^{N/2} |u|^{1/2}/|\sigma_1(t)| \\ |\sigma_2(n)|=q^{N-i}/|\sigma_1(n)| \\ |\sigma_{1,2}(n)|\le |\sqrt{d}|/(\Delta|\sigma_{1,2}(tp)|)}} \mathcal{E} (tnp) \\ 
&= \sum\limits_{\substack{ j\in \mathbb{N}_0 \\ q^{N/2} |u|^{-1/2}/|\sigma_1(t)| < q^j \le q^{N/2} |u|^{1/2}/|\sigma_1(t)| \\ q^j, q^{N-i-j}\le |\sqrt{d}|/(\Delta |\sigma_{1,2}(tp)|)}} V_{j,N-i-j}, 
\end{split} 
\end{equation} 
where 
\begin{equation} \label{Vj1j2defi} 
V_{j_1,j_2}:=\sum\limits_{\substack{ n \in \mathbb{A} \\ |\sigma_1(n)|=q^{j_1} \\ |\sigma_2(n)|= q^{j_2}}} \mathcal{E} (tpn). 
\end{equation} 
We record that the summation conditions $|\sigma_1(n)|=q^{j_1}$ and $|\sigma_2(n)|= q^{j_2}$ together with \eqref{sigma12n} imply the inequalities 
\begin{equation} \label{j1j2estimate} 
q^{(N-i)/2} |u|^{-1} < q^{j_1}, q^{j_2} <q^{(N-i)/2} |u|. 
\end{equation} 
This will be used later on as well. Putting $n= r+s\sqrt{d}$, where $r,s \in \mathbb{F}_q[T]$, the sum in \eqref{Vj1j2defi} becomes 
\begin{equation}\label{V1} 
V_{j_1,j_2} 
= \sum\limits_{\substack{r,s \in \mathbb{F}_q[T]\\ |r+ s\sqrt{d}|=q^{j_1} \\ |r-s\sqrt{d}|= q^{j_2}}} \mathcal{E} (tp(r+s\sqrt{d})) 
= \sum\limits_{s \in \mathbb{F}_q[T]} \quad \sum\limits_{\substack{r \in \mathbb{F}_q[T] \\ r \in \mathbb{S}(-s\sqrt{d},q^{j_1}) \cap \mathbb{S}(s\sqrt{d},q^{j_2})}} \mathcal{E} (tp(r+s\sqrt{d})), 
\end{equation} 
where $ \mathbb{S}(z,q^j):= \{ y \in \mathbb{F}_q(T)_\infty : |z-y| = q^j\}$. 

We have 
\begin{equation*} 
\begin{split} 
\mathcal{E} (tp(r+s\sqrt{d})) &= e\left( \frac{x_1 \sigma_1(tp(r+s\sqrt{d})) - x_2\sigma_2(tp(r+s\sqrt{d}))}{2\sqrt{d}} \right) \\ 
&= e\left( \frac{x_1 \sigma_1(tp)(r+s\sqrt{d}) - x_2\sigma_2(tp)(r-s\sqrt{d})}{2\sqrt{d}} \right) \\ 
&= e(C(tp)r + D(tp) s),\\ 
\end{split} 
\end{equation*} 
where 
\begin{equation} \label{CaDadef} 
C(a):=\frac{ x_1\sigma_1(a) - x_2\sigma_2(a)}{2\sqrt{d}} \quad \mbox{and} \quad D(a):=\frac{ x_1\sigma_1(a)+x_2\sigma_2(a)}{2}. 
\end{equation} 
Therefore the sum in (\ref{V1}) becomes 
\begin{equation*} 
V_{j_1,j_2} = \sum\limits_{s \in \mathbb{F}_q[T]} e(D(tp)s) \quad \sum\limits_{\substack{r \in \mathbb{F}_q[T] \\ r \in \mathbb{S}(-s\sqrt{d},q^{j_1}) \cap \mathbb{S}(s\sqrt{d},q^{j_2})}} e(C(tp) r). 
\end{equation*} 
Using the definition of open balls $B(x,R)$, we write 
\begin{equation*} 
\mathbb{S}(-s\sqrt{d},q^{j_1}) \cap \mathbb{S}(s\sqrt{d},q^{j_2}) =\big( B(-s\sqrt{d}, q^{j_1+1})\setminus B(-s\sqrt{d}, q^{j_1}) \big) \cap \big( B(s\sqrt{d}, q^{j_2+1})\setminus B(s\sqrt{d}, q^{j_2}) \big). 
\end{equation*} 
If $ B(-s\sqrt{d}, q^{j_1+1}) \subseteq B(s\sqrt{d}, q^{j_2+1}),$ then 
\begin{equation*} 
\begin{split} 
\big( B(-s\sqrt{d}, q^{j_1+1})\setminus B(-s\sqrt{d}, q^{j_1}) \big) \cap \big( B(s\sqrt{d}, q^{j_2+1})\setminus B(s\sqrt{d}, q^{j_2}) \big) &= B(-s\sqrt{d}, q^{j_1+1})\setminus B(-s\sqrt{d}, q^{j_1}) \\ 
&= \mathbb{S}(-s\sqrt{d},q^{j_1}).\\ 
\end{split} 
\end{equation*} 
If $ B(s\sqrt{d}, q^{j_2+1}) \subseteq B(-s\sqrt{d}, q^{j_1+1}),$ then 
\begin{equation*} 
\begin{split} 
\big( B(-s\sqrt{d}, q^{j_1+1})\setminus B(-s\sqrt{d}, q^{j_1}) \big) \cap \big( B(s\sqrt{d}, q^{j_2+1})\setminus B(s\sqrt{d}, q^{j_2}) \big) &= B(s\sqrt{d}, q^{j_2+1})\setminus B(s\sqrt{d}, q^{j_2})\\ 
&= \mathbb{S}(s\sqrt{d},q^{j_2}). \\ 
\end{split} 
\end{equation*} 
Moreover, if $B(-s\sqrt{d}, q^{j_1+1})$ and $B(s\sqrt{d}, q^{j_2+1})$ have non-empty intersection, then one of these two balls contains the other. 
Hence, 
\begin{equation} \label{Vj1j2eq} 
V_{j_1,j_2} = V'_{j_1,j_2} + V''_{j_1,j_2} + V'''_{j_1,j_2}, 
\end{equation} 
where 
\begin{equation} \label{V'j1j2} 
V'_{j_1,j_2} := \sum\limits_{\substack{s \in \mathbb{F}_q[T] \\ B(-s\sqrt{d}, q^{j_1+1}) \subsetneqq B(s\sqrt{d}, q^{j_2+1})}} e(D(tp)s) \sum\limits_{\substack{r \in \mathbb{F}_q[T] \\ r \in \mathbb{S}(-s\sqrt{d},q^{j_1})}} e(C(tp)r) , 
\end{equation} 
$$ 
V''_{j_1,j_2}:= \sum\limits_{\substack{s \in \mathbb{F}_q[T] \\ B(-s\sqrt{d}, q^{j_1+1}) = B(s\sqrt{d}, q^{j_2+1})}} e(D(tp)s) \sum\limits_{\substack{r \in \mathbb{F}_q[T] \\ r \in \mathbb{S}(s\sqrt{d},q^{j_2})}} e(C(tp)r) 
$$ 
and 
$$ 
V'''_{j_1,j_2} = \sum\limits_{\substack{s \in \mathbb{F}_q[T] \\ B(s\sqrt{d}, q^{j_2+1}) \subsetneqq B(-s\sqrt{d}, q^{j_1+1})}} e(D(tp)s) \sum\limits_{\substack{r \in \mathbb{F}_q[T] \\ r \in \mathbb{S}(s\sqrt{d},q^{j_2})}} e(C(tp)r). 
$$ 

We write the inner-most sum on the right-hand side of \eqref{V'j1j2} as 
\begin{equation*} 
\begin{split} 
\sum\limits_{\substack{r \in \mathbb{F}_q[T] \\ r \in \mathbb{S}(-s\sqrt{d},q^{j_1})}} e(C(tp)r) &= \sum\limits_{\substack{r \in \mathbb{F}_q[T]}} \chi_{\mathbb{S}(-s\sqrt{d},q^{j_1})}( r) e(C(tp)r). 
\end{split} 
\end{equation*} 
Using the Poisson summation formula, Lemma \ref{Poissonformula}, we obtain 
\begin{equation*} 
\begin{split} 
\sum\limits_{\substack{r \in \mathbb{F}_q[T]}} \chi_{\mathbb{S}(-s\sqrt{d},q^{j_1})}( r) e(C(tp)r) &= \sum\limits_{\substack{r' \in \mathbb{F}_q[T]}} \quad \int\limits_{\mathbb{F}_q(T)_\infty} \chi_{\mathbb{S}(-s\sqrt{d},q^{j_1})}(x) e(C(tp)x)e(-xr') d\mu(x) \\ 
&= \sum\limits_{\substack{r' \in \mathbb{F}_q[T]}} \quad \int\limits_{\mathbb{S}(-s\sqrt{d},q^{j_1})} e((C(tp) - r')x) d\mu(x) \\ 
&= \sum\limits_{\substack{r' \in \mathbb{F}_q[T]}} \quad \int\limits_{\mathbb{S}(0,q^{j_1})} e((C(tp) - r')(x-s\sqrt{d})) d\mu(x)\\ 
&= \sum\limits_{\substack{r' \in \mathbb{F}_q[T]}} e((r'-C(tp))s\sqrt{d}) \quad \int\limits_{\mathbb{S}(0,q^{j_1})} e((C(tp) - r')x) d\mu(x).\\ 
\end{split} 
\end{equation*} 
Using the definition of the Fourier transform and Corollary \ref{Chi1}, we have 
\begin{equation*} 
\begin{split} 
\int\limits_{\mathbb{S}(0,q^{j_1})} e((C(tp) - r')x) d\mu(x) = & \int\limits_{B(0,q^{j_1+1})} e((C(tp) - r')x)d\mu(x) -\int\limits_{B(0,q^{j_1})} e((C(tp) - r')x)d\mu(x) \\ 
= & \int\limits_{B(0,q^{j_1+1})} e((C(tp) - r')x)d\mu(x) -\int\limits_{B(0,q^{j_1})} e((C(tp) - r')x)d\mu(x) \\ 
= & \int\limits_{\mathbb{F}_q(T)_\infty} \chi_{B(0,q^{j_1+1})}(x) e(-(r' -C(tp))x)d\mu(x) \\ & 
-\int\limits_{\mathbb{F}_q(T)_\infty} \chi_{B(0,q^{j_1})}(x) e(-(r' -C(tp))x)d\mu(x) \\ 
= & \hat{\chi}_{B(0,q^{j_1+1})} (r' -C(tp)) - \hat{\chi}_{B(0,q^{j_1})} (r' -C(tp)) \\ 
= & q^{j_1+1}\chi_{B(0,q^{-(j_1+1)})} (r' -C(tp)) -q^{j_1}\chi_{B(0,q^{-j_1})} (r' -C(tp)). 
\end{split} 
\end{equation*} 
It follows that 
\begin{equation*} 
\sum\limits_{\substack{r \in \mathbb{F}_q[T]}} \chi_{\mathbb{S}(-s\sqrt{d},q^{j_1})}( r) e(C(tp)r) = \sum\limits_{\substack{r' \in \mathbb{F}_q[T]}} e((r'-C(tp))s\sqrt{d}) \, G_{(r',j_1)}, 
\end{equation*} 
where 
\begin{equation*} 
G_{(r',k)} := q^{k+1}\chi_{B(0,q^{-(k+1)})} (r' -C(tp)) -q^{k}\chi_{B(0,q^{-k})}(r' -C(tp)). 
\end{equation*} 
Combining the above equations and observing that 
$$ B(-s\sqrt{d}, q^{j_1+1}) \subsetneqq B(s\sqrt{d}, q^{j_2+1}) \Longleftrightarrow |s\sqrt{d}|< q^{j_2+1} \text{ and } j_1 < j_2,$$ 
we obtain 
\begin{equation*} 
V'_{j_1,j_2}= \delta_1(j_1, j_2) \sum\limits_{\substack{r' \in \mathbb{F}_q[T]}} G_{(r',j_1)} \sum\limits_{\substack{s \in \mathbb{F}_q[T] \\ s\in B(0, q^{j_2+1}/|\sqrt{d}|) } } e((D(tp)+ (r'-C(tp))\sqrt{d})s), 
\end{equation*} 
where 
\begin{equation*} 
\delta_1(j_1, j_2) := \begin{cases} 1 & \mbox{ if } j_1< j_2 ,\\ 0 & \mbox{ otherwise.} \end{cases} 
\end{equation*} 
Similarly we deduce that 
\begin{equation*} 
V'''_{j_1,j_2} = \delta_1(j_2, j_1) \sum\limits_{\substack{r' \in \mathbb{F}_q[T]}} G_{(r',j_2)} \sum\limits_{\substack{s \in \mathbb{F}_q[T] \\ s\in B(0, q^{j_1+1}/|\sqrt{d}|) }} e((D(tp)+ (C(tp)-r')\sqrt{d})s) 
\end{equation*} 
and 
\begin{equation*} 
V''_{j_1,j_2}= \delta(j_1, j_2) \sum\limits_{\substack{r' \in \mathbb{F}_q[T]}} G_{(r',j_2)} \sum\limits_{\substack{s \in \mathbb{F}_q[T] \\ s\in B(0, q^{j_2+1}/|\sqrt{d}|) }} e((D(tp)+ (C(tp)-r')\sqrt{d})s), 
\end{equation*} 
where 
\begin{equation*} 
\delta(j_1, j_2) := \begin{cases} 1 & \mbox{ if } j_1= j_2 ,\\ 0 & \mbox{ otherwise.} \end{cases} 
\end{equation*} 
Again using the Poisson summation formula, we get 
\begin{equation*} 
\sum\limits_{\substack{s \in \mathbb{F}_q[T] \\ s\in B(0, q^{j_2+1}/|\sqrt{d}|) } } e((D(tp)+ (r'-C(tp))\sqrt{d})s) = |\sqrt{d}| \sum\limits_{\substack{s' \in \mathbb{F}_q[T]} } q^{j_2+1}\chi_{B(0,q^{-(j_2+1)}|\sqrt{d}|)} (s' - D(tp)+ (C(tp) -r')\sqrt{d}). 
\end{equation*} 
Thus $V'_{j_1,j_2}$ becomes 
\begin{equation} \label{V'eq} 
V'_{j_1,j_2} = \delta_1(j_1, j_2) |\sqrt{d}| \sum\limits_{\substack{r' \in \mathbb{F}_q[T]}} G_{(r',j_1)} \sum\limits_{\substack{s' \in \mathbb{F}_q[T]} } q^{j_2+1}\chi_{B(0,q^{-(j_2+1)}|\sqrt{d}|)} (s' - D(tp)- (r'-C(tp))\sqrt{d}). 
\end{equation} 
Similarly, $V''_{j_1,j_2}$ and $V'''_{j_1,j_2}$ become 
\begin{equation} \label{V''eq} 
V''_{j_1,j_2}= \delta(j_1, j_2) |\sqrt{d}| \sum\limits_{\substack{r' \in \mathbb{F}_q[T]}} G_{(r',j_2)} \sum\limits_{\substack{s' \in \mathbb{F}_q[T]} } q^{j_1+1}\chi_{B(0,q^{-(j_1+1)}|\sqrt{d}|)} (s' - D(tp)+ (r'-C(tp))\sqrt{d}) 
\end{equation} 
and 
\begin{equation} \label{V'''eq} 
V'''_{j_1,j_2}= \delta_1(j_2, j_1) |\sqrt{d}| \sum\limits_{\substack{r' \in \mathbb{F}_q[T]}} G_{(r',j_2)} \sum\limits_{\substack{s' \in \mathbb{F}_q[T]} } q^{j_1+1}\chi_{B(0,q^{-(j_1+1)}|\sqrt{d}|)} (s' - D(tp)+ (r'-C(tp))\sqrt{d}). 
\end{equation} 
From \eqref{V'eq}, we infer the estimate 
\begin{equation} \label{V'ineq} 
\begin{split} 
|V'_{j_1,j_2}| \le & q^{j_1+ j_2 +2} |\sqrt{d}| \sum\limits_{r' \in \mathbb{F}_q[T]} \chi_{B(0,q^{-j_1})}(r' -C(tp)) \sum\limits_{\substack{s' \in \mathbb{F}_q[T]} } \chi_{B(0,q^{-j_1}|\sqrt{d}|)} (s' - D(tp)- (r'-C(tp))\sqrt{d}) \\ 
= & q^{j_1+j_2+2} |\sqrt{d}| \sum\limits_{\substack{r' \in \mathbb{F}_q[T] \\s' \in \mathbb{F}_q[T]} } \chi_{B(0,q^{-j_1})}(r' -C(tp)) \chi_{B(0,q^{-j_1} |\sqrt{d}|)} (s' - D(tp)) \\ 
= & q^{j_1+j_2+2} |\sqrt{d}| \chi_{B(0,q^{-j_1})}( \Vert C(tp) \Vert) \chi_{B(0,q^{-j_1} |\sqrt{d}|)} ( \Vert D(tp) \Vert ). 
\end{split} 
\end{equation} 
Similarly, from \eqref{V''eq} and \eqref{V'''eq}, we infer 
\begin{equation} \label{V''ineq} 
|V''_{j_1,j_2}| \le q^{j_1+j_2+2} \chi_{B(0,q^{-j_1})}( \Vert C(tp) \Vert) \chi_{B(0,q^{-j_1} |\sqrt{d}|)} ( \Vert D(tp) \Vert ) 
\end{equation} 
and 
\begin{equation} \label{V'''ineq} 
|V'''_{j_1,j_2}| \le q^{j_1+j_2+2} \chi_{B(0,q^{-j_2})}( \Vert C(tp) \Vert) \chi_{B(0,q^{-j_2} |\sqrt{d}|)} ( \Vert D(tp) \Vert ). 
\end{equation} 
Combining \eqref{Vj1j2eq}, \eqref{V'ineq}, \eqref{V''ineq} and \eqref{V'''ineq}, we obtain 
\begin{equation} \label{Vj1j2finalbound} 
|V_{j_1,j_2}|\ll q^{j_1+j_2} \left(\chi_{B(0,q^{-j_1})}( \Vert C(tp) \Vert) \chi_{B(0,q^{-j_1} |\sqrt{d}|)} ( \Vert D(tp) \Vert ) +\chi_{B(0,q^{-j_2})}( \Vert C(tp) \Vert) \chi_{B(0,q^{-j_2} |\sqrt{d}|)} ( \Vert D(tp) \Vert ) \right). 
\end{equation} 
Setting $j_1=j$ and $j_2=N-i-j$ and combining \eqref{TI(i)defi}, \eqref{sigma12tp}, \eqref{Vdef}, \eqref{Veq} and \eqref{Vj1j2finalbound}, we deduce that 
\begin{equation}\label{T.11} 
\begin{split} 
T_I(i) \ll & q^{N-i} \quad \sum\limits_{\substack{ t \in \mathbb{A} \\ \mathcal{N}(\mathfrak{t}) =q^i \\ t \in S(\mathfrak{t})}}\ \sum\limits_{\substack{p \in \mathbb{A}\setminus \{0\}\\ |\sigma_{1,2}(tp)|\le |\sqrt{d}u|\Delta^{-1} q^{(i-N)/2}}}\ \sum\limits_{\substack{ j\in \mathbb{N}_0 \\ q^{N/2} |u|^{-1/2}/|\sigma_1(t)| < q^j \le q^{N/2} |u|^{1/2}/|\sigma_1(t)| \\ q^{j}, q^{N-i-j}\le |\sqrt{d}|/(\Delta |\sigma_{1,2}(tp)|)}} \\ 
& \left(\chi_{B(0,q^{-j})}( \Vert C(tp) \Vert) \chi_{B(0,q^{-j} |\sqrt{d}|)} ( \Vert D(tp) \Vert ) +\chi_{B(0,q^{i+j-N})}( \Vert C(tp) \Vert) \chi_{B(0,q^{i+j-N} |\sqrt{d}|)} ( \Vert D(tp) \Vert ) \right). 
\end{split} 
\end{equation} 
Next, we write $k=tp$ and note that 
$$ 
B(0,q^{-j}),B(0,q^{i+j-N}),B(0,q^{-j} |\sqrt{d}|),B(0,q^{i+j-N} |\sqrt{d}|) 
\subseteq B(0,|u\sqrt{d}|q^{(i-N)/2}) 
$$ 
by \eqref{j1j2estimate}, thus getting 
\begin{equation}\label{T1(i).25} 
T_I(i) \ll q^{N-i} \sum\limits_{\substack{k \in \mathbb{A}\setminus \{0\}\\ |\sigma_{1,2}(k)|< \Delta^{-1}R(i) }} \chi_{B(0, R(i))} \left( \left| \left| \frac{ x_1\sigma_1(k) - x_2\sigma_2(k)}{2\sqrt{d}} \right| \right| \right) \chi_{B(0, R(i))} \left( \left| \left| \frac{ x_1\sigma_1(k)+x_2\sigma_2(k)}{2} \right| \right| \right), 
\end{equation} 
where 
\begin{equation} \label{Ridef} 
R(i):=|u\sqrt{d}|q^{(i-N)/2}. 
\end{equation} 

\section{Counting Process} 
Assume we have simultaneous Diophantine approximation of the pair $(x_1, x_2)$ by a pair $(\sigma_1(\theta),\sigma_2(\theta_2))$ with $\theta\in K=\mathbb{F}_q(T)(\sqrt{d})$ satisfying 
\begin{equation} \label{Diophant} 
|x_i - \sigma_i(\theta)|\le \Delta_1 \text{ for } i=1,2 
\end{equation} 
for a suitable $\Delta_1>0$. Then \eqref{T1(i).25} gives us 
\begin{equation} \label{sum T_I(i) C1} 
\begin{split} 
T_I(i) \ll & q^{N-i} \sum\limits_{\substack{k \in \mathbb{A}\setminus \{0\}\\ |\sigma_{1,2}(k)|< \Delta^{-1}R(i)}} \chi_{B(0, \tilde{R}(i))} \left( \left| \left| \frac{ \sigma_1(k\theta) - \sigma_2(k\theta)}{2\sqrt{d}} \right| \right| \right) \chi_{B(0, \tilde{R}(i))} \left( \left| \left| \frac{ \sigma_1(k\theta)+\sigma_2(k \theta)}{2} \right| \right| \right), \\ 
= & q^{N-i} 
\sum\limits_{\substack{k \in \mathbb{A}\setminus \{0\}\\ |\sigma_{1,2}(k)|< \Delta^{-1}R(i)}} \chi_{B(0, \tilde{R}(i))} (||\Im(k\theta)||) \chi_{B(0, \tilde{R}(i))} (||\Re(k\theta)||), 
\end{split} 
\end{equation} 
where 
\begin{equation} \label{tildeRidef} 
\tilde{R}(i)= (1+ \Delta_1 \Delta^{-1})R(i) 
\end{equation} 
and $\Re(f)$ and $\Im(f)$ are defined as in \eqref{RealImaginary}. 
We write 
\begin{equation} \label{Defi of theta} 
\theta = \frac{u+v\sqrt{d}}{f+g\sqrt{d}} 
\end{equation} 
with $u,v,f,g \in \mathbb{F}_q[T]$, where $u+v\sqrt{d}$ and $f+g\sqrt{d}$ are relatively prime in $\mathbb{A}$. By Theorem \ref{Dirichlet1}, there are infinitely many $\theta\in K$ such that \eqref{Diophant} holds with 
\begin{equation} \label{Defi of |W|} 
\Delta_1:= \frac{c_1}{W}, 
\end{equation} 
where 
$$ 
W:= \mathcal{N}(f+g\sqrt{d}) 
$$ 
and $c_1>0$ is a suitable constant. 

Next, we break the sum in the last line of \eqref{sum T_I(i) C1} into 
$$ 
O\left(\frac{R(i)^2}{\Delta^2 W} +1\right) 
$$ 
subsums of the form 
\begin{equation} \label{broken sum in T_1} 
\Sigma(A_1,A_2)= \sum\limits_{\substack{k \in \mathbb{A}\setminus \{0\}\\ \sigma_1(k) \in B\left(A_1, \sqrt{\mathcal{N}(f+g\sqrt{d})}\right) \\ \sigma_2(k) \in B\left(A_2, \sqrt{\mathcal{N}(f+g\sqrt{d})}\right) }} \chi_{B(0, \tilde{R}(i))} (||\Im(k\theta)||) \chi_{B(0, \tilde{R}(i))} (||\Re(k\theta)||) 
\end{equation} 
for suitable $A_1, A_2\in \mathbb{F}_q(t)_{\infty}$. We claim that for $k$ running over $\mathbb{A}$ and satisfying 
$$ 
\sigma_1(k) \in B\left(A_1, \sqrt{\mathcal{N}(f+g\sqrt{d})}\right) 
$$ 
and 
$$ 
\sigma_2(k) \in B\left(A_2, \sqrt{\mathcal{N}(f+g\sqrt{d})}\right), 
$$ 
$k\theta (f+g\sqrt{d})=k(u+v\sqrt{d})$ covers every residue class $\bmod (f+g\sqrt{d})$ at most once. This can be seen as follows. If 
$$ 
k_1(u+v\sqrt{d}) \equiv k_2(u+v\sqrt{d}) \bmod (f+g\sqrt{d}), 
$$ 
then 
$$ 
k_1- k_2 \equiv 0 \bmod (f+g\sqrt{d}) 
$$ 
since $u+v\sqrt{d}$ and $f+g\sqrt{d}$ are relatively prime. 
We have 
$$ 
|\sigma_1(k_1 - k_2)|= |\sigma_1(k_1) -\sigma_1(k_2)| \leq \max \{ |\sigma_1(k_1)-A_1|, |\sigma_1(k_2)- A_1| \} < \sqrt{\mathcal{N}(f+g\sqrt{d})}. 
$$ 
Similarly, 
$$ 
|\sigma_2(k_1 - k_2)| < \sqrt{\mathcal{N}(f+g\sqrt{d})}. 
$$ 
Thus 
$$ 
\mathcal{N}(k_1-k_2) = |\sigma_1(k_1 - k_2)||\sigma_2(k_1 - k_2)| < \mathcal{N}(f+g\sqrt{d}). 
$$ 
However, if $k_1 \neq k_2$, then $\mathcal{N}(f+g\sqrt{d}) \mid \mathcal{N}(k_1-k_2)$ which implies $\mathcal{N}(f+g\sqrt{d}) \leq \mathcal{N}(k_1-k_2)$. Thus, we have reached a contradiction. Therefore $k_1 = k_2$, which establishes the above claim. It follows that 
\begin{equation} \label{SigmaA1A2esti} 
\begin{split} 
\Sigma(A_1,A_2) \leq & \sharp \left\{ n \in \mathbb{A}: \mathfrak{R}\left(\frac{n}{f+g\sqrt{d}}\right), \mathfrak{I}\left(\frac{n}{f+g\sqrt{d}}\right) \in B(0, \tilde{R}(i)) \right\}\\ 
\leq & \sharp \left\{ n \in \mathbb{A} : \sigma_1\left(\frac{n}{f+g\sqrt{d}}\right), \sigma_2\left(\frac{n}{f+g\sqrt{d}}\right) \in B(0, |\sqrt{d}| \tilde{R}(i)) \right\}. 
\end{split} 
\end{equation} 

The set 
\begin{equation*} 
\Lambda(f+ g\sqrt{d}):= \left\{ \left(\sigma_1\left( \frac{n}{f+g\sqrt{d}}\right), \sigma_2\left(\frac{n}{f+g\sqrt{d}}\right)\right)\in \mathbb{F}_q(t)_{\infty} : n \in \mathbb{A} \right\} 
\end{equation*} 
forms a lattice in $K_\infty^2$ with covolume $|\sqrt{d}|/W$. With this notation, \eqref{SigmaA1A2esti} simplifies into 
\begin{equation} \label{SigmaA1A2esti2} 
\begin{split} 
\Sigma(A_1,A_2) \leq \sharp \left(\Lambda(f+ g\sqrt{d}) \cap B_2(0, |\sqrt{d}| \tilde{R}(i))\right). 
\end{split} 
\end{equation} 

\subsection{Counting lattice points} 
Now we aim to establish the bound 
\begin{equation}\label{count Lattice point} 
\sharp \left(\Lambda(f+ g\sqrt{d}) \cap B_2(0, |\sqrt{d}| \tilde{R}(i))\right)\ll W\tilde{R}(i)^2 +1. 
\end{equation} 
Here the crucial point is that without loss of generality, $f+g\sqrt{d}$ can be chosen in such a way that 
\begin{equation} \label{fgassump} 
|\sigma_1(f+g\sqrt{d})| \asymp \sqrt{\mathcal{N}(f+g\sqrt{d})} \quad \mbox{and} \quad |\sigma_2(f+g\sqrt{d})| \asymp \sqrt{\mathcal{N}(f+g\sqrt{d})}. 
\end{equation} 
This situation can be achieved by multiplying both the denominator and numerator in the fraction on the right-hand side of \eqref{Defi of theta} by a suitable unit. 

We have 
\begin{equation} \label{count lattice 2} 
\begin{split} 
& \left(\Lambda(f+ g\sqrt{d}) \cap B_2(0, |\sqrt{d}| \tilde{R}(i))\right)\\ 
= & \left\{ n \in \mathbb{A} : \frac{1}{W} \left( \sigma_1(n)\sigma_2(f+g\sqrt{d}), \sigma_2(n)\sigma_1(f+g\sqrt{d}) \right) \in B_2(0, |\sqrt{d}| \tilde{R}(i)) \right\} \\ 
= & \left\{ (a,b)\in \mathbb{F}_q[T]^2 : \left( (a+b\sqrt{d})(f-g\sqrt{d}), (a-b\sqrt{d})(f+g\sqrt{d}) \right) \in B_2\left( 0, W|\sqrt{d}| \tilde{R}(i)\right) \right\}.
\end{split} 
\end{equation} 
Since we may choose $f+g\sqrt{d}$ in such a way that \eqref{fgassump} is satisfied, the condition 
$$ 
\left( (a+b\sqrt{d})(f-g\sqrt{d}), (a-b\sqrt{d})(f+g\sqrt{d}) \right) \in B_2\left( 0, W|\sqrt{d}| \tilde{R}(i)\right) 
$$ 
implies 
\begin{equation*} 
|a+b\sqrt{d}| \ll_d \sqrt{W}\tilde{R}(i)\quad \mbox{and} \quad 
|a-b\sqrt{d}| \ll_d \sqrt{W}\tilde{R}(i) 
\end{equation*} 
and further 
\begin{equation*} 
|a| \ll_d \sqrt{W}\tilde{R}(i) \quad \mbox{and} \quad 
|b| \ll_d \sqrt{W}\tilde{R}(i). 
\end{equation*} 
From \eqref{count lattice 2} and the above, we deduce that 
\begin{equation*} 
\begin{split} 
\sharp \left(\Lambda(f+ g\sqrt{d}) \cap B_2(0, |\sqrt{d}| \tilde{R}(i))\right) 
\ll & \sharp \left\{ (a,b)\in \mathbb{F}_q[T]^2 : |a|,|b|\le C(d)\sqrt{W}\tilde{R}(i) \right\} \\ \ll & W\tilde{R}(i)^2+1, 
\end{split} 
\end{equation*} 
where $C(d)$ is a suitable constant depending only on $d$. This establishes the claim \eqref{count Lattice point}. 

\subsection{Estimating $T_{I}(i)$} 
Combining \eqref{sum T_I(i) C1}, \eqref{SigmaA1A2esti2} and \eqref{count Lattice point}, we get 
\begin{equation} \label{Estimate 1 of T_1} 
T_I(i) \ll q^{N-i} \left(\frac{R(i)^2}{\Delta^2W} +1\right)\left(W\tilde{R}(i)^2 +1\right). 
\end{equation} 
Recalling \eqref{Ridef} and \eqref{tildeRidef}, we deduce that 
\begin{equation*} 
T_I(i) \ll q^{N-i} \left(\left(q^{N-i}\Delta^2W\right)^{-1} +1\right) \left(\left(1+\Delta_1^2\Delta^{-2}\right)Wq^{i-N} +1 \right). 
\end{equation*} 
Further, recalling \eqref{deltadefi} and \eqref{Defi of |W|}, this implies 
\begin{equation}\label{T_I(i) bound final} 
\frac{ q^N \Delta^2 \log_q(\alpha\beta)}{|\sqrt{d}|(q-1)^2}\cdot T_I(i) \ll \log_q(\alpha\beta) \left( \left( 1+ \delta^{-2}W^{-2}q^N \right) q^i + 
\left(W^{-1}q^N + \delta^2W\right)+ \delta^2 q^{N-i}\right), 
\end{equation} 
which, by \eqref{TIequation}, is the contribution of $i$ to $T_I$. 

\subsection{Excluding small $i$'s} 
If $i$ is small then the last term in the estimate \eqref{T_I(i) bound final} is of size about $\log_q(\alpha\beta) \delta^2 q^{N}$, which is too large to get a non-trivial estimate for $T_I$. In this subsection, we will see that small $i$'s can be excluded, i.e. the sum $T_I(i)$ is actually empty for them. To this end, taking \eqref{sum T_I(i) C1} into account, it suffices to bound the sizes of $\max\{||\mathfrak{R}(k\theta)||,|| \mathfrak{I}(k\theta)||\}$ from below. 

Recall \eqref{Defi of theta}. Let the lattice $\tilde{\Lambda}(f+g\sqrt{d})$ be defined by 
$$ 
\tilde{\Lambda}(f+g\sqrt{d}):=\left\{\left(\Re\left(\frac{n}{f+g\sqrt{d}}\right),\Im\left(\frac{n}{f+g\sqrt{d}}\right)\right)\in \mathbb{F}_q(t)_{\infty}^2 : n\in \mathbb{A}\right\}. 
$$ 
For simplicity, we write $\Lambda=\Lambda(f+g\sqrt{d})$ and $\tilde{\Lambda}=\tilde{\Lambda}(f+g\sqrt{d})$ throughout the following. 
Clearly, if $k\in \mathbb{A}$, then $(\Re(k\theta),\Im(k\theta))\in \tilde{\Lambda}$. Moreover, $F_q[T]^2\subseteq \tilde{\Lambda}$. Further, if 
\begin{equation} \label{kcondi} 
0<\mathcal{N}(k)<W=\mathcal{N}(f+g\sqrt{d}), 
\end{equation} 
then $k\theta\not\in F_q[T]^2$. Under this condition, it follows that 
$$ 
\max\{||\mathfrak{R}(k\theta)||,|| \mathfrak{I}(k\theta)||\}\ge \min\limits_{(\alpha,\beta)\in \tilde{\Lambda}\setminus \{0\}} \max\{|\alpha|,|\beta|\}. 
$$ 
To estimate the right-hand side, we note that 
\begin{equation*} 
\begin{split} 
\min\limits_{(\alpha,\beta)\in\tilde{\Lambda}\setminus \{0\}} \max\{|\alpha|,|\beta|\}= & 
\min\limits_{(\mu,\nu)\in \Lambda\setminus \{0\}} \max\left\{\left|\frac{\mu+\nu}{2}\right|,\left|\frac{\mu-\nu}{2\sqrt{d}}\right|\right\}\\ 
\ge & \frac{1}{|\sqrt{d}|} \min\limits_{(\mu,\nu)\in \Lambda\setminus \{0\}} \max\left\{\left|\frac{\mu+\nu}{2}\right|,\left|\frac{\mu-\nu}{2}\right|\right\}\\ 
= & \frac{1}{|\sqrt{d}|} \min\limits_{(\mu,\nu)\in \Lambda\setminus \{0\}} \max\left\{|\mu|,|\nu|\right\}\\ 
= & \frac{1}{|\sqrt{d}|W}\min\limits_{a+b\sqrt{d} \in \mathbb{A}\setminus\{0\}} \max \left\{ \left| (f-g\sqrt{d})(a+b \sqrt{d}) \right|, \left|(f+g\sqrt{d})(a-b\sqrt{d}) \right| \right\}\\ 
\gg & \frac{1}{\sqrt{W}}\min\limits_{a+b\sqrt{d} \in \mathbb{A}\setminus\{0\}} \max \left\{ \left| a+b \sqrt{d} \right|, \left| a-b\sqrt{d} \right| \right\} \\ 
\gg & \frac{1}{\sqrt{W}}\min\limits_{(a,b) \in \mathbb{F}_q[T]\setminus\{0\}} \max \left\{ \left| a\right|, \left| b\right| \right\}\\ 
\gg & \frac{1}{\sqrt{W}}, 
\end{split} 
\end{equation*} 
where we use \eqref{fgassump}. In view of \eqref{sum T_I(i) C1}, we deduce that $T_I(i)=0$ if 
$$ 
\Delta^{-1} R(i)\le W^{1/2} \quad \mbox{and} \quad \tilde{R}(i) < c_2W^{-1/2} 
$$ 
for a suitable constant $c_2>0$. Considering \eqref{deltadefi}, \eqref{Ridef}, \eqref{tildeRidef} and \eqref{Defi of |W|}, this is the case if $q^i\le J$, where 
$$ 
J:= c_3\min\left\{\delta^2 W,W^{-1}q^N\right\} 
$$ 
for a suitable constant $c_3>0.$ 

\subsection{Estimating $T_I$} 
Using \eqref{TIequation}, \eqref{T_I(i) bound final}, the above considerations and 
$\log_q(\alpha\beta)\ll N$, we get 
\begin{equation} \label{T_1 final sum} 
\begin{split} 
T_I \ll & N \sum\limits_{\substack{i\in \mathbb{N}_0 \\ J < q^i \le \alpha \beta }} 
\left( \left( 1+ \delta^{-2}W^{-2}q^N \right) q^i + 
\left(W^{-1}q^N + \delta^2W\right)+ \delta^2 q^{N-i}\right) \\ 
\ll & N^2\left(\left( 1+ \delta^{-2}W^{-2}q^N \right) \alpha\beta+ W^{-1}q^N + \delta^2W\right). 
\end{split} 
\end{equation}

\section{Transformation of the type-II sum} 
Recall \eqref{typeIIform}. Using the definitions of $\mathcal{F}(\mathfrak{a})$ in \eqref{Fdefi} in the case $\mathcal{N}(\mathfrak{a})\not=q^N$ and $\mathcal{S}(\mathfrak{a})$ in \eqref{cardS} together with the identities \eqref{anote} and \eqref{Sidentity}, it follows that 
\begin{equation} \label{TIIsplit} 
T_{II} \le 
\frac{q^N \Delta^2}{|\sqrt{d}| (q-1)^2} \cdot \sum\limits_{\substack{i\in \mathbb{N}_0 \\ \beta < q^i \le q^N/ \alpha}} T_{II}(i) 
\end{equation} 
with 
$$ 
T_{II}(i):= \Big| \sum\limits_{\substack{m \in \mathbb{A} \\ \mathcal{N}(\mathfrak{m})= q^i \\ m \in S(\mathfrak{m}) }} \Lambda(\mathfrak{m}) \quad \sum\limits_{\substack{ n \in \mathbb{A} \\ \mathcal{N}(\mathfrak{n}) = q^{N-i} \\ mn \in S(\mathfrak{mn}) }} H(\mathfrak{n}) \quad \sum\limits_{p\in \mathbb{A}\setminus \{0\}} \mathcal{E} (mn p) E(mn p) \Big|, 
$$ 
where $E(l)$ and $\mathcal{E}(l)$ are defined as in \eqref{EEdefi}. Using \eqref{Eident} with $t$ replaced by $m$, we deduce that 
\begin{equation*} 
\begin{split} 
T_{II}(i) = & \Big| \sum\limits_{\substack{m \in \mathbb{A} \\ \mathcal{N}(\mathfrak{m})= q^i \\ m \in S(\mathfrak{m}) }} \Lambda(\mathfrak{m}) \sum\limits_{\substack{n \in \mathbb{A} \\ \mathcal{N}(\mathfrak{n}) = q^{N-i} \\ mn \in S(\mathfrak{mn}) \\ |\sigma_{1,2}(n)|\le |\sqrt{d}|/(\Delta|\sigma_{1,2}(mp)|) }} H(\mathfrak{n}) \quad \sum\limits_{\substack{p\in \mathbb{A}\setminus \{0\} }} \mathcal{E} (mn p)\Big| \\ 
\le & \sum\limits_{\substack{p\in \mathbb{A}\setminus \{0\} }} \,\, \sum\limits_{\substack{m \in \mathbb{A} \\ \mathcal{N}(\mathfrak{m})= q^i \\ m \in S(\mathfrak{m}) }} \Lambda(\mathfrak{m})\cdot \Big| \sum\limits_{\substack{n\in \mathbb{A} \\ \mathcal{N}(\mathfrak{n}) = q^{N-i} \\ mn \in S(\mathfrak{mn})\\ |\sigma_{1,2}(n)|\le |\sqrt{d}|/(\Delta|\sigma_{1,2}(mp)|) }} H(\mathfrak{n}) \mathcal{E} (mn p) \Big|. 
\end{split} 
\end{equation*} 

As $mn\in S(\mathfrak{mn})$ and $\mathcal{N}(\mathfrak{mn})=q^N,$ we have 
\begin{equation}\label{SigmaNM} 
q^{N/2}|u|^{-1/2} < |\sigma_{1,2}(mn)|< q^{N/2}|u|^{1/2}. 
\end{equation} 
Similarly, as $m\in S(\mathfrak{m})$ and $\mathcal{N}(\mathfrak{m})=q^i$, we have 
\begin{equation}\label{SigmaM} 
q^{i/2}|u|^{-1/2} < |\sigma_{1,2}(m)|< q^{i/2}|u|^{1/2}. 
\end{equation} 
From \eqref{SigmaNM}, \eqref{SigmaM} and $|\sigma_1(m)\sigma_2(m)|=q^i$, we get 
\begin{equation*} 
q^{(N-i)/2}|u|^{-1} < |\sigma_{1,2}(n)|< q^{(N-i)/2}|u|. 
\end{equation*} 
Writing $k=mp$, we deduce that 
\begin{equation*} 
\begin{split} 
T_{II}(i) & \le \sum\limits_{\substack{k\in \mathbb{A}\setminus \{0\} \\ |\sigma_{1,2}(k)|\le \Delta^{-1} |\sqrt{d}u| q^{(i-N)/2}}} \,\, a_k \cdot \Big| \sum\limits_{\substack{n\in \mathbb{A} \\ \mathcal{N}(\mathfrak{n}) = q^{N-i} \\ q^{(N-i)/2}|u|^{-1} < |\sigma_{1,2}(n)|< q^{(N-i)/2}|u| \\ |\sigma_{1,2}(n)|\le |\sqrt{d}|/(\Delta|\sigma_{1,2}(k)|)}} H(\mathfrak{n}) \mathcal{E} (nk ) \Big|,\\ 
\end{split} 
\end{equation*} 
where 
$$ 
a_k := \sum\limits_{\substack{(m, p) \in \mathbb{A}^2 \\ mp=k \\ m \in S(\mathfrak{m}) }} \Lambda(\mathfrak{m}). 
$$ 
Applying the Cauchy- Schwarz inequality, we obtain 
\begin{equation}\label{Sum2.64} 
T_{II}(i)^2 \le A(i) B(i), 
\end{equation} 
where 
$$ 
A(i):= \sum\limits_{\substack{k\in \mathbb{A}\setminus \{0\} \\ |\sigma_{1,2}(k)|\le \Delta^{-1} |\sqrt{d} u| q^{(i-N)/2}}} |a_k|^2 
$$ 
and 
$$ 
B(i):=\sum\limits_{\substack{k\in \mathbb{A}\\ |\sigma_{1,2}(k)|\le \Delta^{-1} |\sqrt{d} u| q^{(i-N)/2}}} \Big| \sum\limits_{\substack{n\in \mathbb{A} \\ \mathcal{N}(\mathfrak{n}) = q^{N-i} \\ q^{(N-i)/2}|u|^{-1} < |\sigma_{1,2}(n)|< q^{(N-i)/2}|u| \\ |\sigma_{1,2}(n)|\le |\sqrt{d}|/(\Delta|\sigma_{1,2}(k)|)}} H(\mathfrak{n}) \mathcal{E} (nk ) \Big|^2. 
$$ 
Note that we have now included the contribution of $k=0$, which is possible due to non-negativity of the modulus on the right-hand side. 

In the following, we bound the term $A(i)$. We may express $a_k$ explicitly as 
$$ 
a_k = \sum\limits_{\substack{m \in \mathbb{A} \\ m \mid k \\ m \in S(\mathfrak{m}) }} \Lambda(\mathfrak{m}) = (q-1) \log_q(\mathcal{N}(k)), 
$$ 
which implies 
\begin{equation}\label{Bound of A} 
A(i) \ll \log_q^2(\mathcal{N}(k)) \left(\Delta^{-1} q^{(i-N)/2}\right)^2 
=N^2 \Delta^{-2} q^{i-N}. 
\end{equation} 

Next we will treat the term $B(i)$. Expanding the modulus square, exchanging summations and using the bound 
$$ 
H(\mathfrak{n}) \ll \sum\limits_{\substack{\mathfrak{d}|\mathfrak{n}\\ }} 1\ll 2^{\log_q \mathcal{N}(n)}, 
$$ 
we obtain 
\begin{equation} \label{TIIpreliminary} 
B(i) \ll 2^{2(N-i)} \sum\limits_{\substack{n_1, n_2 \in \mathbb{A} \\ q^{(N-i)/2}|u|^{-1}<|\sigma_{1,2}(n_{1})|< q^{(N-i)/2}|u|\\ 
q^{(N-i)/2}|u|^{-1}<|\sigma_{1,2}(n_{2})|< q^{(N-i)/2}|u|}} \, \, |U(n_1,n_2,i)|, 
\end{equation} 
where 
$$ 
U(n_1,n_2,i):=\sum\limits_{\substack{k\in \mathbb{A}\\ |\sigma_{1}(k)|\le L_1(n_1,n_2,k)\\ |\sigma_{2}(k)|\le L_2(n_1,n_2,k)}} \mathcal{E} (nk) 
$$ 
with $n=n_1-n_2$ and 
$$ 
L_j(n_1,n_2,i):=\min\left\{\Delta^{-1} |\sqrt{d} u| q^{(i-N)/2}, |\sqrt{d}|/(\Delta|\sigma_{j}(n_1)|),|\sqrt{d}|/(\Delta|\sigma_{j}(n_2)|)\right\} \quad \mbox{ for } j=1,2. 
$$ 
The sum $U(n_1,n_2,i)$ above can be treated in a similar way as the sum $V_{j_1,j_2}$ defined in \eqref{Vj1j2defi} was treated in section \ref{transformationTI}. (In fact, the calculations become easier because the $\sigma_{1,2}(k)$'s are contained in balls rather than spheres.) Similarly as \eqref{Vj1j2finalbound}, we thus get 
\begin{equation*} 
U(n_1,n_2,i)\ll L_1L_2 \left(\chi_{B(0,L_1^{-1})}( \Vert C(n) \Vert) \chi_{B(0,L_1^{-1} |\sqrt{d}|)} ( \Vert D(n) \Vert ) +\chi_{B(0,L_2^{-1})}( \Vert C(n) \Vert) \chi_{B(0,L_2^{-1} |\sqrt{d}|)} ( \Vert D(n) \Vert ) \right) 
\end{equation*} 
with $L_j=L_j(n_1,n_2,i)$ for $j=1,2$. Writing 
\begin{equation} \label{lidefi} 
l(i):=\Delta |\sqrt{d}u|q^{(N-i)/2}, 
\end{equation} 
we observe that 
$$ 
L_1L_2\le l(i)^{-2} \quad \mbox{and}\quad L_j^{-1}|\sqrt{d}|\le l(i), 
$$ 
where we use the summation conditions on $n_1$ and $n_2$ in \eqref{TIIpreliminary}. It follows that 
$$ 
U(n_1,n_2,i)\ll l(i)^{-2} \chi_{B(0,l(i))}( \Vert C(n) \Vert) \chi_{B(0,l(i))} ( \Vert D(n) \Vert). 
$$ 
Plugging this into \eqref{TIIpreliminary}, we see that 
\begin{equation*} 
B(i) \ll 2^{2(N-i)} \Delta^{-2} q^{i-N}\sum\limits_{\substack{n \in \mathbb{A}\\ |\sigma_{1,2}(n)|< q^{(N-i)/2}|u|}} b_n\cdot \chi_{B(0,l(i))}( \Vert C(n) \Vert) \chi_{B(0,l(i))} ( \Vert D(n) \Vert), 
\end{equation*} 
where 
$$ 
b_n:= \sum\limits_{\substack{n_1, n_2 \in \mathbb{A} \\ n=n_1-n_2\\ q^{(N-i)/2}|u|^{-1}<|\sigma_{1,2}(n_{1})|< q^{(N-i)/2}|u|\\ 
q^{(N-i)/2}|u|^{-1}<|\sigma_{1,2}(n_{2})|< q^{(N-i)/2}|u|}} 1. 
$$ 
Clearly, $b_n\ll q^{N-i}$. Hence, recalling the definitions of $C(a)$ and $D(a)$ in \eqref{CaDadef}, we arrive at the bound 
\begin{equation} \label{8.T2(i)a} 
\begin{split} 
B(i) \ll 4^N\Delta^{-2} 
\sum\limits_{\substack{n \in \mathbb{A}\\ |\sigma_{1,2}(n)|< q^{(N-i)/2}|u|}} 
\chi_{B(0,l(i))}\left( \left| \left| \frac{ x_1\sigma_1(n) - x_2\sigma_2(n)}{2\sqrt{d}} \right| \right| \right) \chi_{B(0,l(i))}\left( \left| \left| \frac{ x_1\sigma_1(n)+x_2\sigma_2(n)}{2} \right| \right| \right). 
\end{split} 
\end{equation} 

\section{Estimation of the type-II sum} 
We notice that the expressions in \eqref{8.T2(i)a} and \eqref{T1(i).25} are of a similar shape. Therefore we can directly apply our method in section 7. In place of \eqref{Estimate 1 of T_1}, we thus get 
\begin{equation} \label{Bound of B} 
B(i) \ll 4^N\Delta^{-2} \left( \frac{l(i)^2}{\Delta^2 W} +1\right) \left(\Tilde{l}(i)^2W+1\right), 
\end{equation} 
where 
\begin{equation} \label{8.L(i)*} 
\Tilde{l}(i):= \left(1+\Delta_1\Delta^{-1}\right)l(i)\ll \left(1+(\Delta W)^{-1}\right)l(i). 
\end{equation} 
Combing \eqref{Sum2.64}, \eqref{Bound of A} and \eqref{Bound of B}, using the definitions of $l(i)$ and $\Tilde{l}(i)$ in \eqref{lidefi} and \eqref{8.L(i)*}, and taking the square root, we deduce that 
\begin{equation} \label{first} 
T_{II}(i) \ll 2^N N\Delta^{-2} \left(W^{-1/2} +q^{(i-N)/2}\right) \left(\Delta q^{(N-i)/2}W^{1/2}+q^{(N-i)/2}W^{-1/2}+1\right). 
\end{equation} 
Reversing the roles of $m$ and $n$ in the whole process of estimating $T_{II}(i)$, we arrive at the same bound with $q^i$ in place of $q^{N-i}$, i.e. 
\begin{equation} \label{second} 
T_{II}(i) \ll 2^NN\Delta^{-2} \left(W^{-1/2} +q^{i/2}\right) \left(\Delta q^{i/2}W^{1/2}+q^{i/2}W^{-1/2}+1\right). 
\end{equation} 
We shall use \eqref{first} if $i\le N/2$ and \eqref{second} if $i> N/2$. Accordingly, we split the sum over $i$ in \eqref{TIIsplit} into two parts, getting 
\begin{equation} \label{TIIsplit2} 
T_{II} \ll T'+T'', 
\end{equation} 
where 
\begin{equation} \label{T'} 
T':=q^{N} \Delta^2 
\sum\limits_{\substack{i\in \mathbb{N}_0 \\ \beta < q^i \le q^{N/2}}} T_{II}(i) 
\end{equation} 
and 
\begin{equation} \label{T''} 
T'':=q^{N} \Delta^2 
\sum\limits_{\substack{i\in \mathbb{N}_0 \\ q^{N/2} < q^i \le q^{N}/\alpha }} T_{II}(i). 
\end{equation} 
Employing \eqref{first}, we have 
\begin{equation} \label{T'esti} 
\begin{split} 
T'\ll & 2^N Nq^N \sum\limits_{\substack{i\in \mathbb{N}_0 \\ \beta < q^i \le q^{N/2}}}\left(W^{-1/2} +q^{(i-N)/2}\right) \left(\Delta q^{(N-i)/2}W^{1/2}+q^{(N-i)/2}W^{-1/2}+1\right) \\ 
\ll & 2^N N q^N \sum\limits_{\substack{i\in \mathbb{N}_0 \\ \beta < q^i \le q^{N/2}}}\left(\Delta q^{N/2-i/2}+q^{N/2-i/2}W^{-1}+W^{-1/2}+\Delta W^{1/2}+q^{i/2-N/2}\right)\\ 
\ll & 2^N N^2 \left(\Delta q^{3N/2}\beta^{-1/2}+q^{3N/2}W^{-1}\beta^{-1/2}+q^NW^{-1/2}+\Delta q^N W^{1/2}+q^{3N/4}\right)\\ 
\ll & 2^N N^2 \left(\delta q^{N}\beta^{-1/2}+q^{3N/2}W^{-1}\beta^{-1/2}+q^NW^{-1/2}+\delta q^{N/2} W^{1/2}+q^{3N/4}\right), 
\end{split} 
\end{equation} 
where for the last line, we have used the relation $\Delta=\delta/q^{N/2}$. 
Similarly, employing \eqref{second}, we have 
\begin{equation} \label{T''esti} 
\begin{split} 
T''\ll & 2^N Nq^N \sum\limits_{\substack{i\in \mathbb{N}_0 \\ q^{N/2} < q^i \le q^{N}/\alpha }}\left(W^{-1/2} +q^{-i/2}\right) \left(\Delta q^{i/2}W^{1/2}+q^{i/2}W^{-1/2}+1\right) \\ 
\ll & 2^N N q^N \sum\limits_{\substack{i\in \mathbb{N}_0 \\ q^{N/2} < q^i \le q^{N}/\alpha}}\left(\Delta q^{i/2}+q^{i/2}W^{-1}+W^{-1/2}+\Delta W^{1/2}+q^{-i/2}\right)\\ 
\ll & 2^N N^2 \left(\Delta q^{3N/2}\alpha^{-1/2}+q^{3N/2}W^{-1}\alpha^{-1/2}+q^NW^{-1/2}+\Delta q^N W^{1/2}+q^{3N/4}\right)\\ 
\ll & 2^N N^2 \left(\delta q^{N}\alpha^{-1/2}+q^{3N/2}W^{-1}\alpha^{-1/2}+q^NW^{-1/2}+\delta q^{N/2} W^{1/2}+q^{3N/4}\right). 
\end{split} 
\end{equation} 
Putting \eqref{TIIsplit2}, \eqref{T'esti} and \eqref{T''esti} together, we get 
\begin{equation} \label{8. T2 final} 
T_{II} \ll 2^N N^2 \left(\delta q^{N}\min\{\alpha,\beta\}^{-1/2}+q^{3N/2}W^{-1}\min\{\alpha,\beta\}^{-1/2}+q^NW^{-1/2}+\delta q^{N/2} W^{1/2}+q^{3N/4}\right). 
\end{equation} 

\section{Proof of Theorem \ref{Main Theorem 2}} 
We recall where the parameter $W$ comes from: In section 7, we simultaneously approximated the pair $(x_1, x_2)$ by a pair $(\sigma_1(\theta),\sigma_2(\theta_2))$ with 
$$ 
\theta = \frac{u+v\sqrt{d}}{f+g\sqrt{d}}\in K, \quad (u+v\sqrt{d},f+g\sqrt{d})\approx 1. 
$$ 
The parameter $W$ is the norm of the denominator above, i.e. $W= \mathcal{N}(f+g\sqrt{d})$. There are infinitely many choices for $\theta$ satisfying our hypotheses, giving rise to an infinite increasing sequence of possible $W$'s. Now we pick such a $W$ and choose $N$ depending on $W$ and the parameter $\delta<1$ as 
$$ 
N:=2\lceil \log_q(\delta W) \rceil 
$$ 
so that 
$$ 
W \asymp \delta^{-1}q^{N/2}. 
$$ 
Under this choice, \eqref{8. T2 final} collapses into 
\begin{equation*} 
T_{II} \ll 2^N N^2 \left(\delta q^{N}\min\{\alpha,\beta\}^{-1/2}+q^{3N/4}\right) 
\end{equation*} 
and \eqref{T_1 final sum} collapses into 
\begin{equation*} \label{9. T1 final2} 
T_I \ll N^2(\alpha\beta +\delta q^{N/2}). 
\end{equation*} 
Fixing 
$$ 
\alpha=\beta:=\delta^{2}q^{N/2} 
$$ 
and recalling \eqref{diffi}, it follows that 
$$ 
\Tilde{\mathcal{T}}(N)-\mathcal{T}(N)\ll 2^N N^2 \left(\delta^4q^{N}+\delta q^{N/2}+q^{3N/4}\right)\ll_{\varepsilon} q^{(\log_q 2+\varepsilon)N}\left(\delta^4q^{N}+\delta q^{N/2}+q^{3N/4}\right). 
$$ 
Hence, we have 
\begin{equation*} 
\tilde{\mathcal{T}}(N)-\mathcal{T}(N) \ll_{\varepsilon} \delta^2 q^{(1-\varepsilon)N} 
\end{equation*} 
if $q>2^{12}$ and 
\begin{equation*} 
q^{-(1/8 -\log_q \sqrt{2}-2\varepsilon)N}\le \delta\le q^{-(\log_q 2 +2\varepsilon)N}. 
\end{equation*} 
This establishes Theorem \ref{Main Theorem 2}. 

\section{Proof of Theorem \ref{main theorem}} 
Recall the definitions of $\mathcal{T}(N)$ and 
$\tilde{\mathcal{T}}(N)$ in 
\eqref{def of T(N)} and \eqref{Def of til(T(N))}. Using the prime ideal theorem for function fields, we have 
$$ 
{\mathcal{T}}(N)\sim \frac{q^N\delta^2}{|\sqrt{d}|}. 
$$ 
Combining this with Theorem \ref{Main Theorem 2}, we deduce that under the assumptions in this theorem, there exists an infinite sequence of positive integers $N$ such that 
$$ 
\tilde{\mathcal{T}}(N)\sim \frac{q^N\delta^2}{|\sqrt{d}|} 
$$ 
whenever $\delta$ satisfies \eqref{deltaassump}. Fix a member $N$ of this sequence and $\delta$ satisfying \eqref{deltaassump}. By a standard argument, the contribution of prime powers $\mathfrak{a}=\mathfrak{p}^k$ with $k\ge 2$ to $\tilde{\mathcal{T}}(N)$ is bounded by $O\left(N^2 q^{N/2}\right)$, where we use the fact that $\omega(\mathfrak{a})$ is bounded. Taking the size of $\delta$ into account, it follows that the remaining contribution satisfies the same asymptotic 
$$ 
\sum\limits_{\substack{ \mathfrak{p} \in \mathcal{P}\\ \mathcal{N}(\mathfrak{p})= q^N }}\Lambda(\mathfrak{p})\omega(\mathfrak{p}) \sim \frac{q^N\delta^2}{|\sqrt{d}|}, 
$$ 
where $\mathcal{P}$ is the set of prime ideals in $\mathbb{A}$. In particular, there exists a prime ideal $\mathfrak{p}$ with norm $\mathcal{N}(\mathfrak{p})=q^N$ such that $\omega(\mathfrak{p})>0$. Using the definition of $\omega(\mathfrak{a})$ in \eqref{Fdef}, we conclude that there exists a prime element $\pi\in \mathbb{A}$ with norm $\mathcal{N}(\pi)=q^N$ such that 
$$ 
\left|x_i-\frac{\sigma_i(p)}{\sigma_i(\pi)}\right| \le \frac{\delta}{q^{N/2}} \quad \mbox{ for } i=1,2. 
$$ 
This implies \eqref{diorequired} upon choosing $\delta$ as small as possible, i.e. 
$$ 
\delta:=q^{-(1/8 -\log_q \sqrt{2}-2\varepsilon)N}, 
$$ 
and changing $2\varepsilon$ into $\varepsilon$. Hence, Theorem \ref{main theorem} is established. 

\section{Simplification of the treatment of real quadratic number fields} 
Our general method in this article is similar to that in \cite{BM}, where the same problem was handled for real quadratic number fields, except for the following two points. Firstly, we here use a version of Vaughan's identity instead of Harman's sieve, which appeared easier to us in this context. Secondly, and more importantly, our counting process in section 7 differs from that in \cite[sections 7 and 8]{BM} and is considerably simpler. In \cite{BM}, the counting problem was interpreted as solving a system of two linear congruences for pairs of integers lying in certain ranges. To get an acceptable bound for the number of solutions, the authors made use of the fact that a certain quantity is a root of a quadratic congruence (see \cite[first congruence in section 8]{BM}) and then approximated the quotient of this root and the modulus by a rational number with smaller denominator. To this end, they used a result of Hooley. This forced them to introduce a notion of ``good" and ``bad" pairs $(x_1,x_2)$ in their paper, where good pairs satisfied a certain Diophantine constraint. Their final result was of comparable strenght as ours for good pairs but weaker for bad pairs. Luckily, as they demonstrated, almost every pair $(x_1,x_2)$ is good. They even gave an explicit construction of good pairs. 

In this article, we handled the counting problem more directly by breaking the summation over $k$ into subsums, where the lattice points $(\sigma_1(k),\sigma_2(k))$ lie in balls of radius $\sqrt{\mathcal{N}(f+g\sqrt{d})}$ (see equation \eqref{broken sum in T_1}). It remained to count these lattice points in the said balls, which was a comparibly easy task. This could have been done in an analogous manner in \cite{BM} as well and would have avoided introducing the notion of good and bad pairs. Here we don't carry out this process again but just state that it leads to the following result for the number field case, which sharpens \cite[Theorem 5]{BM} and is an analogue of our Theorem \ref{main theorem}, formulated in terms of prime ideals. 
Here we note that the additional term $\log_q \sqrt{2}$ in the exponent on the right-hand side of \eqref{diorequired} becomes less than $\varepsilon$ if $q$ is large enough, namely if $q>\exp(\sqrt{2}/\varepsilon)$. 

\begin{theorem} 
Assume that $\mathbb{Q}(\sqrt{d})$ has class number 1, where $d$ is a positive square-free integer satisfying $d \equiv 3 \bmod{4}$. Let $\varepsilon$ be any positive real number. Suppose further that $(x_1, x_2) \in \mathbb{R}^2\setminus \sigma(K)$. Then there exist infinitely many non-zero prime ideals $\mathfrak{p}$ in the ring $\mathcal{O}$ of integers of $\mathbb{Q}(\sqrt{d})$ such that 
$$ 
\left| x_i-\frac{\sigma_i(p)}{\sigma_i(q)}\right| \le \mathcal{N}(\mathfrak{p})^{-1/2-1/8+\varepsilon} \quad \mbox{for } i=1,2 
$$ 
for some generator $q$ of $\mathfrak{p}$ and $p\in \mathcal{O}$. 
\end{theorem}

\end{document}